\documentclass[12pt]{amsart}

\usepackage{amssymb,amsbsy,amsmath,amsfonts,amssymb,amscd}
\usepackage{latexsym}
\usepackage{graphics}
\usepackage{color}

\input xy
\xyoption{all}

\newcommand\sC{{\mathcal C}}

\newcommand\sL{{\mathcal L}}

\newcommand\sK{{\mathcal K}}

\newcommand\Ga{\Gamma}

\newcommand\De{\Delta}
\newcommand\ga{\gamma}

\DeclareMathOperator{\Fix}{Fix}

\DeclareMathOperator{\Hom}{Hom}
\DeclareMathOperator{\Alb}{Alb}
\DeclareMathOperator{\Jac}{Jac}

\DeclareMathOperator{\Def}{Def}
\newcommand{\CC}{\ensuremath{\mathbb{C}}}
\newcommand{\RR}{\ensuremath{\mathbb{R}}}
\newcommand{\ZZ}{\ensuremath{\mathbb{Z}}}
\newcommand{\QQ}{\ensuremath{\mathbb{Q}}}

\newcommand{\hol}{\ensuremath{\mathcal{O}}}

\newcommand{\PP}{\ensuremath{\mathbb{P}}}

\newcommand{\ra}{\ensuremath{\rightarrow}}

\def\eea{\end{eqnarray*}}
\def\bea{\begin{eqnarray*}}

\newcommand\dual{\mathrel{\raise3pt\hbox{$\underline{\mathrm{\thinspace d
\thinspace}}$}}}

\newcommand\QED{\ifhmode\unskip\nobreak\fi\quad {\rm Q.E.D.}} 
\newcommand\qe{\ifhmode\unskip\nobreak\fi\quad $\Box$}       

\newcommand\onto{\twoheadrightarrow}
\def\BOX{\hfill\lower.5\baselineskip\hbox{$\Box$}}

\newtheorem{theo}{Theorem}[section]
\newtheorem{remarkk}[theo]{Remark}
\newenvironment{rem}{\begin{remarkk}\rm}{\end{remarkk}}

\newtheorem{defin}[theo]{Definition}
\newenvironment{definition}{\begin{defin}\rm}{\end{defin}}

\newtheorem{prop}[theo] {Proposition}
\newtheorem{cor}[theo]{Corollary}
\newtheorem{lemma}[theo]{Lemma}
\newtheorem{example}[theo]{Example}

\newtheorem{claim}[theo]{Claim}

\DeclareMathOperator{\Aut}{Aut}
\DeclareMathOperator{\Aff}{Aff}

\DeclareMathOperator{\im}{Im}
\DeclareMathOperator{\Tors}{Tors}
\DeclareMathOperator{\Torsion}{Torsion}
\DeclareMathOperator{\Irr}{Irr}
\DeclareMathOperator{\GL}{GL}

\newcommand{\Proof}{{\it Proof. }}
\begin{document}

\title[Inoue type  manifolds and surfaces  ]{Inoue type manifolds and Inoue surfaces: 
  a connected component of the moduli space of surfaces  with $K^2 = 7$,
$p_g=0$.}
\author{I. Bauer, F. Catanese}

\thanks{The present work took place in the realm of the DFG
Forschergruppe 790 "Classification of algebraic surfaces and
compact complex manifolds".}

\date{\today}

\maketitle

{\bf Dedicated to Gerard  van der Geer on the occasion of his 60-th birthday.\\
Dedicato a Gherardo in occasione del 60-esimo genetliaco.\\
Gerard  zu seinem 60. Geburtstag gewidmet.}

 \bigskip

   {\bf ABSTRACT.}
    We show that a family of minimal surfaces of general type with   $p_g = 0, K^2=7$,
    constructed by Inoue in 1994,    is indeed a connected component of the moduli space:
    indeed that any surface which is homotopically equivalent to an Inoue surface belongs to the Inoue family.

    The ideas used in order to show this result motivate us to give a new definition of varieties,
    which we propose to call Inoue-type manifolds: these are obtained as quotients 
    $ \hat{X} / G$, where $ \hat{X} $ is an ample divisor in a $K(\Ga, 1)$
    projective manifold $Z$, and $G$ is a finite group acting freely on $ \hat{X} $.
    For these type of manifolds we prove a similar  theorem to the above, even if weaker,  that manifolds 
    homotopically equivalent to  Inoue-type manifolds are again  Inoue-type manifolds.
    
     \bigskip

   {\bf SUNTO.}
  Lo scopo di questo lavoro e' duplice: da una parte vogliamo qui  mostrare che una famiglia di superficie minimali di tipo   generale con genere geometrico nullo, e genere lineare $p_1 = 8$,  
    costruite dal signor  Inoue nel 1994,   formano una componente connessa dello spazio dei  moduli.
    Anzi, piu' precisamente, mostriamo che ogni 
   superficie omotopicamente  equivalente ad una  superficie  di Inoue appartiene alla suddetta famiglia.

    Le  idee su cui si appoggiano le tecniche dimostrative sono di carattere assai generale e ci inducono 
    a proporre come oggetto di studio una classe di variet\'a proiettive che vogliamo qui chiamare variet\'a di tipo Inoue.
    
    Queste variet\'a vengono definite come quozienti $ \hat{X} / G$ (per la azione di un gruppo finito $G$
    che agisca liberamente su $ \hat{X}$), dove 
 $ \hat{X} $ \'e un  divisore ampio in una variet\'a proiettiva $Z$ che sia uno spazio di Eilenberg MacLane  $K(\Ga, 1)$.
 Per queste variet\'a siamo in grado di mostrare un teorema analogo al precedente, anche se  piu' debole,
 che stabilisce che variet\'a
  omotopicamente  equivalenti a variet\'a di tipo Inoue sono anche esse variet\'a di tipo Inoue.
 
 \bigskip

\section*{Introduction}Minimal surfaces of general type with $p_g(S)
= 0$ have invariants
$p_g(S)  =  q (S)= 0, 1 \leq K_S^2  \leq 9$, and, for each value of $
y \in \{1,2, \dots, 9\}$, such surfaces
with $ K_S^2 = y$  yield  a finite number
of irreducible components of the Gieseker moduli space of surfaces of
general type $\mathfrak M^{can}_{1,y}$.

After the first surfaces of general type with $p_g = q = 0$ were
constructed in the 1930' s by   Luigi Campedelli and by Lucien
Godeaux (cf. \cite{Cam},
\cite{god}) there was  in the 1970's a big revival of interest in the
construction of these surfaces,
as documented by a long and influential survey paper  written by
Dolgachev (\cite{Dolgachev}). 

The Bloch conjecture and
differential topological questions raised by Donaldson Theory were a
further reason for the remarkable and ongoing  interest
about surfaces of general type with
$p_g = 0$, and  we refer to
\cite{survey} for an update  about   recent  important progress on
the topic, and about  the state of the art.

Looking at the tables 1--3 of \cite{survey} one finds  it striking
that for the value $K_S^2= 7$ there is only one  known family of
such surfaces of general type. This family was constructed by Inoue
(cf. \cite{inoue}).
Further interest concerning this family comes from the problem raised
in \cite{survey}, and concerning the finiteness or not
of the fundamental group of  surfaces with $p_g = 0$. Indeed this paper was
motivated by the observation that Inoue's surfaces have a
``big'' fundamental group. In fact, the fundamental group of an Inoue
surface with $p_g = 0$ and $K_S^2 =7$ sits
in an extension ($\pi_g$ denotes the fundamental group of a compact
curve of genus $g$):
$$ 1 \rightarrow \pi_{5} \times \mathbb{Z}^4 \rightarrow \pi_1(S)
\rightarrow (\mathbb{Z}/2\mathbb{Z})^5 \rightarrow 1.
$$

This extension is given geometrically, i.e., stems from our observation that
an Inoue surface $S$ admits an unramified
$(\mathbb{Z}/2\mathbb{Z})^5$ - Galois covering $\hat{S}$ which is an
ample divisor in $E_1 \times E_2 \times D$, where
$E_1, E_2$ are elliptic curves and $D$ is a compact curve of genus
$5$; this made us believe
   that the topological type of an Inoue surface determines an irreducible
connected component of the moduli space (a phenomenon similar to the
one which was already observed   in \cite{keumnaie},
\cite{burniat1} \cite{coughlanchan}).

The following is  one of the main  results of this paper:

\begin{theo}\label{main}\
\begin{enumerate}
\item Let $S'$ be a smooth complex projective surface which is
homotopically equivalent to an Inoue surface (with $K^2 = 7$ and
$p_g=0$). Then $S'$ is an Inoue  surface.
\item The connected component of the Gieseker moduli space
$\mathfrak M^{can}_{1, 7}$ corresponding to Inoue surfaces is
irreducible, generically smooth, normal and unirational  of dimension  4.

Moreover, each Inoue surface $S$ has ample canonical divisor\footnote{This is proven by Inoue, 
see page 318 of \cite{inoue}.} and the base $\Def (S)$  of the Kuranishi family of $S$  is smooth.
\item
Finally, the first homology group of an Inoue surface equals $$\ZZ/4 \ZZ \oplus (\ZZ/2 \ZZ)^4.$$

\end{enumerate}
\end{theo}

Since this theorem is similar in flavour to other results that we 
mentioned above,
the main purpose of this paper is not only to give a more general proof, but also to  set up the
stage for  the investigation and search for a new class of varieties, 
which we propose to call
Inoue-type varieties.

\begin{defin}
We define a complex projective manifold $X$ to be an {\bf Inoue-type manifold} if
\begin{enumerate}
\item
$ dim (X) \geq 2$;
\item
there is a finite group $G$ and a Galois unramified covering $
\hat{X} \ra X$ with group $G$,
(hence $ X = \hat{X} / G$) so that
\item
$ \hat{X}$ is an ample divisor inside a $K(\Ga, 1)$-projective manifold
$Z$ (hence by Lefschetz $\pi_1 ( \hat{X}) \cong \pi_1 (Z)
\cong \Ga$) and moreover
\item
the action of $G$ on $ \hat{X}$   yields a faithful action on $\pi_1
( \hat{X}) \cong \Ga$:
in other words the exact sequence
$$ 1 \ra \Ga  \cong \pi_1 ( \hat{X}) \ra \pi_1 ( X) \ra G \ra 1$$
gives an injection $ G \ra Out (\Ga)$, defined by conjugation;
\item
the action of $G$ on $ \hat{X}$   is induced by an action on $Z$.
\end{enumerate}
Similarly one defines the notion of an  {\bf Inoue-type variety}, by
requiring the same properties
for a variety $X$ with canonical singularities.
\end{defin}

We should warn the reader that our approach was inspired by, but is different from the
original construction of Inoue, who considers
hypersurfaces and complete intersections (of not necessarily ample
divisors) in a product of elliptic curves.
In fact, the change of point of view in the particular case  of  an Inoue surface with 
$K^2 = 7, p_g = 0$ produces a different realization: instead of Inoue's   original
realization as
 a complete intersection of two hypersurfaces of multidegrees $(2,2,2,0)$ and
$(0,0,2,2)$ in a product of 4 elliptic curves,
  we view the same surface as a hypersurface of multidegrees
$(2,2,4)$ in the product
$E_1 \times E_2 \times D$ of two elliptic curves with a curve $D$ of genus $5$.

One can see that our definition, although imposing a strong restriction on $X$, 
is not yet satisfactory in order to obtain some
weak rigidity result (of the type of theorems 4.13 and 4.14 of 
\cite{isogenous}, amended in \cite{annals}, theorem 1.3).
Some hypotheses must be made on the fundamental group $\Ga$ of $Z$,
for instance the most interesting case
is the one where $Z$ is a product of Abelian varieties, curves, and other locally symmetric varieties with
ample canonical bundle.

\begin{defin}\label{SIT}
   We shall say  that an  Inoue-type manifold $X$  is 

\begin{enumerate}
\item
  a {\bf SIT : = special Inoue type manifold}
if  moreover
$$ Z = (A_1 \times \dots \times A_r) \times  (C_1 \times \dots \times
C_h) \times (M_1 \times \dots \times
M_s)$$ where each
$A_i$ is an Abelian variety,  each $C_j$ is a curve of genus $ g_j \geq 2$,
and $M_i$ is  a
compact quotient of an irreducible bounded symmetric domain of 
dimension at least 2 by a
torsion free subgroup.
\item
  a {\bf CIT : = 
classical Inoue type manifold}
if  moreover

$ Z = (A_1 \times \dots \times A_r) \times  (C_1 \times \dots \times
C_h) $ where each
$A_i$ is an Abelian variety,  each $C_j$ is a curve of genus $ g_j \geq 2$.
\item
a special Inoue type manifold  is said to be
a 

{\bf diagonal SIT manifold : = diagonal
special   Inoue type manifold }
if  moreover:
\begin{itemize}
\item the action of $G$ on $ \hat{X}$   is induced by a  diagonal 
action on $Z$, i.e.,

$$
(I) \ \  G \subset    \prod_{i=1}^r \Aut(A_i)
\times
\prod_{j=1}^h\Aut(C_j) \times \prod_{l=1}^s\Aut(M_l)
$$

\item
and furthermore:

(II) 
the  faithful action on $\pi_1 ( \hat{X}) \cong \Ga$,
induced by conjugation by lifts of elements of $G$ in  the exact sequence
\begin{equation}\label{prodFG}
 1 \ra \Ga= \Pi_{i=1}^r (\Lambda_i) \times   \Pi_{j=1}^h (\pi_{g_j}) 
\times \Pi_{l=1}^s
(\pi_1 (M_l))
\ra \pi_1 ( X) \ra G \ra 1
\end{equation}
(observe that each factor $\Lambda_i$, resp. $\pi_{g_j}, \pi_1 (M_l)$ is 
normal), has the Schur property
$$(SP)  \Hom (V_i, V_j)^{G} = 0,\forall i \neq j,   $$
where $V_j : = \Lambda_j\otimes \QQ$ (it  suffices then to verify  that 
for each $\Lambda_i$ there is a subgroup $H_i$ of $G$ for which $ \Hom (V_i, V_j)^{H_i} = 0, \forall j \neq i  $).
\end{itemize}
\item
similarly we define a {\bf diagonal CIT manifold : = diagonal
classical  Inoue type manifold }

\end{enumerate}
We can define analogous  notions for    Inoue  type varieties
  $X$ with canonical singularities.
\end{defin}

Property (SP) plays an important role in order to show that an
Abelian variety  with such a $G$-action on its fundamental group
must split as a product.

There is however  a  big difference between the curve and locally symmetric
factors  on one side and the
Abelian variety factors on the other.  Namely: for curves we have weak rigidity, i.e., 
the action of $G$ on $\pi_{g_j}$ determines a connected family of
curves; for compact free quotients of bounded symmetric domains of 
dimension $\geq 2$ we have strong rigidity, i.e., 
the action on the fundamental group determines uniquely the 
holomorphic action; 
for Abelian varieties it is not necessarily so.

Hence, in order to hope for weak rigidity results, one has to 
introduce  a further invariant, called Hodge type (see section 1).

We can now state our main general result:

\begin{theo}\label{special diagonal}\
  Let $X$ be a diagonal SIT (special  Inoue type) manifold, and let $X'$ be a
projective manifold
with the same fundamental group as $X$, which moreover either

(1)
is homotopically equivalent to $X$;

  or satisfies the following weaker property:

(2) let $\hat{X'}$ be the corresponding unramified covering of $X'$ and
identify both
fundamental groups $\pi_1 (\hat{X'}) = \pi_1 (\hat{X})= \Ga $ . Then 
we require that

 {\bf (SAME HOMOLOGY)} there is an isomorphism $\Psi : H_*(\hat{X'}, \ZZ) \cong
H_*(\hat{X}, \ZZ)  $
of homology groups which is compatible with the homomorphisms
$$ u \colon H_*(\hat{X}, \ZZ)  \ra   H_*(\Ga , \ZZ)   ,  u'   \colon
 H_*(\hat{X'}, \ZZ)  \ra   H_*(\Ga , \ZZ)  ,   $$
i.e., $u \circ \Psi = u'$.

Under these assumptions, we have, setting $ W : = \hat{X'}$, that

\begin{itemize}
\item
$X'=  W / G$ where $W$ admits a generically finite
morphism 
$f : W \ra Z'$, and where
\item
$Z'$ is also a $K(\Ga, 1)$ projective manifold,
of the form $ Z' = (A'_1 \times \dots \times A'_r) \times
(C'_1 \times \dots \times  C'_h) \times (M'_1 \times \dots \times
M'_s)$.

Moreover here $M'_i$ is either $M_i$ or its complex conjugate,
and the product decomposition corresponds to the product decomposition
(\ref{prodFG}) of the fundamental group of $Z$.
\item
The  image cohomology class  $f_*([W])$ corresponds, up to sign,
to the cohomology class of $\hat{X}$.
\item
The morphism $f$ is
finite if $n = dim X$ is odd, and moreover it is
generically injective if 

(**) the cohomology class of $\hat{X}$ is indivisble,
or if every strictly submultiple  cohomology class cannot be represented by an effective
$G$-invariant divisor  on any pair  $(Z', G)$ homotopically equivalent to $(Z,G)$.
\item
 $f$ is an embedding if moreover $K_{X'}$ is ample and

(***)  $\ \ \ K_{X'}^n = K_{X}^n$.\footnote{ This last property for
algebraic surfaces follows automatically from homotopy invariance.}
\end{itemize}
In particular, if $K_{X'}$ is ample and (**) and (***) hold, also $X'$ is a diagonal SIT (special  Inoue type) manifold.

A similar  conclusion holds under the alternative assumption that  the isomorphism of homology groups sends the canonical class of $W$
to the one of $\hat{X}$: then  $X'$ is a diagonal SIT (special  Inoue type) variety.

\end{theo}

Hypothesis  (1) in the previous theorem allows to derive the
conclusion that also
$W : = \hat{X'} $ admits a holomorphic map $f'$ to a complex manifold $Z'$
with the same structure as $Z$, while
 hypotheses  (2) and following ensure  that the morphism is birational to its image,
and the class of the image
divisor $f' (\hat{X'} )$ corresponds to $\pm$ the one of $\hat{X} $  under
the identification
$$ H_*(Z' , \ZZ) \cong  H_*(\Ga , \ZZ) \cong   H_*(Z , \ZZ).$$

Since $K_{X'}$ is ample, one uses (***) to conclude that $f'$ is an isomorphism with its image.

The next question is weak rigidity, which amounts to the existence of a
connected complex manifold
parametrizing all such maps (or the  complex conjugate).
Here  several ingredients come into play, namely, firstly the Hodge type,
secondly a fine analysis of
the structure of the action of $G$ on $Z$, in particular concerning the existence of hypersurfaces on which $G$ acts freely. Finally, one would have to see
whether the family of the invariant effective divisors
thus obtained is parametrized by a connected family: this also
requires further work which we do not undertake here except that for the case of Inoue surfaces.

It would  take  long to analyse here the most general situation, yet
 there is an even  more general
situation worth to be investigated. This is the case of {\bf orbifold
Inoue type varieties},
where the action of $G$ is no longer free: this situation is
especially appealing from the point of view of the construction of new
interesting examples.

The paper is organized as follows:  in the first section we deal with general Inoue type manifolds,
establish the first  general properties
of Inoue type manifolds, and prove our main theorem \ref{special diagonal}. Further, more complete, results
dealing with weak rigidity will be given elsewhere.

Section two is devoted to preliminaries, for instance  on curves of 
genus 5 admitting symmetries by  $(\ZZ/2 \ZZ)^4$.
This is an important background for the construction of Inoue 
surfaces with $K^2_S = 7$ and $p_g (S)= 0$,
which is explained in detail in section three.
The end of section three is then devoted to a new result, namely, the calculation of the
first homology group of an Inoue surface, which is shown to be equal to the  group $(\ZZ/4 \ZZ)\oplus (\ZZ/2 \ZZ)^4$.

Section four proves the main result on Inoue surfaces with $K^2_S = 
7$ and $p_g (S)= 0$.

Finally, section five is devoted to showing that the the moduli space 
of Inoue surfaces is generically smooth: this is achieved
by looking at another  realization as  bidouble covers of a four nodal cubic.
\section{ Inoue type manifolds and varieties}

Assume that $X = \hat{X} / G$ is an Inoue-type manifold, so that 
there is an isomorphism
$\pi_1 (\hat{X} ) \cong \pi_1 (Z )= : \Ga $, by virtue of Lefschetz' theorem.

In general, if $W$ is another K\"ahler manifold with $\pi_1 (W  ) 
\cong \pi_1 (Z )= : \Ga $,
we would like to be able to assert that there exists a holomorphic 
map $ f : W \ra Z'$ where $Z'$ is another $K(\Ga,1)$ manifold and 
where  $f_* : \pi_1 (W  ) \ra \pi_1 (Z' ) \cong \Ga$ realizes the above isomorphism.

 This is for instance the case if $Z$ is a
compact quotient of an irreducible bounded symmetric domain of 
dimension at least 2 by a
torsion free subgroup; this follows by combining the results of Eells and Sampson  (\cite{es})
proving the existence of a harmonic map in each homotopy class of maps $ f : W \ra Z$, since $Z$
has negative curvature, with the results of Siu (\cite{siu1} and 
\cite{siu2}), showing the complex dianalyticity of the resulting harmonic map  (i.e., the map $f$  is holomorphic or antiholomorphic)
in the case where $f_* : \pi_1 (W  ) \ra \pi_1 (Z' ) \cong \Ga$ is an isomorphism, since then the differential of $f$ has rank $\geq 4$
(as a linear map of real vector spaces).
Observe that in this case $Z'$ is  $ Z$, or the complex conjugate $\bar{Z}$.
 
Another case is the case where $Z$ is a compact curve of genus $g 
\geq 2$. In this case,
after several results by Siu, Beauville and others (see \cite{albgentype}), a
  simple criterion was shown to be the existence of a surjection $\pi_1 (W  ) \ra 
\pi_1 (Z )$  with finitely generated kernel (see
\cite{barlotti} and \cite{cime2}, Theorem 5.14), an assumption which holds true in our situation.

It is on the above grounds that we restricted ourselves to special Inoue type manifolds
(the diagonality assumption is only a simplifying assumption).

Let us prove the first general result, namely, theorem \ref{special diagonal}.

\Proof {\em of Theorem \ref{special diagonal}.}

{ \bf Step 1}

The first step consists in showing that  $W : = \hat{X'}$ admits a holomorphic
mapping to a manifold $Z'$ of the above type $ Z' = (A'_1 \times 
\dots \times A'_r) \times
(C'_1 \times \dots \times  C'_h) \times (M'_1 \times \dots \times
M'_s)$, where $M'_i$ is either $M_i$ or its complex conjugate.

First of all, by the cited results of Siu and others (\cite{siu1}, 
\cite{siu2},\cite{barlotti},  \cite{cime2}, Theorem 5.14), $W$ admits a 
holomorphic map to
a product manifold $$Z_2'\times Z'_3 = (C'_1 \times \dots \times 
C'_h) \times (M'_1 \times
\dots
\times  M'_s).$$

Look now at the  Albanese variety $\Alb (W)$ of the K\"ahler manifold 
$W$, whose fundamental
group is the quotient of the Abelianization of $\Ga = \pi_1 (Z)$ by its torsion subgroup. 

Write the fundamental group of
$\Alb (W)$ as the first homology group of $A  \times  Z_2  \times Z_3$,
i.e., as 
$$ H_1 ( \Alb (W)) = \Lambda \oplus H_1 ( Z_2, \ZZ) \oplus (H_1 ( Z_3, \ZZ)/ \Torsion),$$ 
($\Alb (Z_2)$ is the product of Jacobians
$J : = (\Jac(C_1) \times \dots \times  \Jac(C_h))  $).

Since however, by the universal property,  $\Alb (W)$ has a holomorphic map to 
$$B' : = \Alb (Z'_2) \times \Alb (Z'_3),$$
inducing a splitting of the lattice  $H_1 (\Alb (W), \ZZ) = \Lambda \oplus 
H_1 (B', \ZZ)$,
it follows that $\Alb(W)$ splits as $A' \times B'$.

Now, we want to show that the Abelian variety $A'$ ($W$ is assumed to 
be a projective
manifold) splits as desired.  This is in turn a consequence of assumption (3) in definition \ref{SIT}.
In fact, the group $G$ acts on the Abelian variety $A'$ as a group of 
biholomorphisms,
hence it acts on $\Lambda \otimes \RR$ commuting with multiplication by 
$\sqrt{-1}$.
Hence multiplication by $\sqrt{-1}$ is an isomorphism of $G$ representations, and then (3)
 implies  that $\Lambda_i\otimes \RR$ is stable by  multiplication by 
$\sqrt{-1}$;
whence $\Lambda_i\otimes \RR$ generates a subtorus $A'_i$. Finally, $A'$ splits because $\Lambda$ is
the direct sum of the sublattices $\Lambda_i$. We are through with the 
proof of step 1.

{ \bf Step 2}

Consider now the holomorphic map $ f \colon W \ra Z'$. We shall show 
that  the image
$ W' : = f (W)$ is indeed a divisor in $Z'$.

This follows from the assumption  (SAME HOMOLOGY).

\begin{lemma}\label{cohalg}
Assume that $W$ is a K\"ahler manifold, such that
\begin{itemize}
\item[i)] there is an isomorphism of
fundamental groups $\pi_1 (W) = \pi_1 (\hat{X})= \Ga$, where $ 
\hat{X}$ is a smooth ample
divisor in a $K(\Ga, 1)$ complex projective manifold $Z$;
\item[ii)] there exists a holomorphic map  $ f : W \ra Z'$, where $Z'$ is another 
$K(\Ga, 1)$ complex  manifold,  such that $f_* : \pi_1 (W) \ra  \pi_1 (Z') = \Ga$ is an 
isomorphism, and moreover
\item[iii)]  {\bf (SAME HOMOLOGY)} there is an isomorphism $\Psi : H_*( W, \ZZ) \cong
H_*(\hat{X}, \ZZ)  $
of homology groups which is compatible with the homomorphisms
$$ u \colon H_*(\hat{X}, \ZZ)  \ra   H_*(\Ga , \ZZ)   ,  u'   \colon
 H_*( W, \ZZ)  \ra   H_*(\Ga , \ZZ)  ,   $$
i.e., $u \circ \Psi = u'$.

\end{itemize}
Then $f$ is a generically finite  morphism of $W$ into $Z'$, and 
the  cohomology class  $f_* ([W])$  in
$$ H^*(Z', \ZZ) = H^*(Z, \ZZ)
=H^*(\Ga, \ZZ)  $$ corresponds to $\pm 1$  the one  of $ \hat{X}$.
\end{lemma}

{\it Proof of the Lemma.}
 We can identify under $\Psi : H_*(W, \ZZ) \cong
H_*(\hat{X}, \ZZ)  $ the homology groups of $W$ and $ \hat{X}$, 
and then the image
in the homology groups of
$ H_*(Z', \ZZ) = H_*(Z, \ZZ) =H_*(\Ga , \ZZ)  $ is the same.

We apply the above consideration to the fundamental classes of the oriented manifolds $W$ and $ \hat{X}$,
which are generators of the infinite cyclic top degree homology groups $H_{2n}(W, \ZZ)$,
respectively $H_{2n}( \hat{X}, \ZZ)$.

This implies a fortiori that $ f \colon W \ra Z'$ is generically finite:
since then the homology class  $f_* ([W])$ (which we identify to a cohomology class  by virtue of Poincar\'e duality)
equals  the class of $ \hat{X}$, up to sign.

\qed

{\bf Step 3}

We claim that  $ f \colon W \ra Z'$ is generically 1-1 onto its image $W'$.

Let in fact  $d$ be the degree of $ f \colon W \ra W'$.

Then   $f_* ([W]) = d [W']$, hence if the class of 
$\hat{X}$ is indivisible, then obviously $d=1$.

Otherwise, observe that the divisor $W'$ is an effective $G$-invariant divisor and use our assumption.

{\bf Step 4}

We  established that  $f$ is birational onto its image $W'$, hence it is a 
desingularization of $W'$.

We use now adjunction. We claim that, since $K_W$ is nef, there exists an effective 
divisor $\mathfrak A$,
called adjunction divisor, such that
$$K_W = f^* (K_{Z'} + W') - \mathfrak A.$$

This can be shown by taking the Stein factorization
$$ W \ra W^N \ra W', $$
where $W^N$ is the normalization of $W'$.

Let $\sC $ be the conductor ideal  $\mathcal{H}om (p_*\hol_{W^N}, \hol_{W'} )$
viewed as an ideal $\sC \subset \hol_{W^N}$; then the Zariski  canonical divisor 
of $W^N$ satisfies 
$$K_{W^N} =   p^*   (K_{W'}) - C =   p^*   (K_{Z'} + W') - C$$
where $C$ is the Weil divisor associated to the conductor ideal (the equality on the Gorenstein locus
of $W^N$ is shown for instance in \cite{bucharest}, then it suffices to take the direct image from the open set to
the whole of $W^N$).

In turn, we would have in general 
$K_W = h^* (K_{W^N } ) - \mathfrak B $, with $\mathfrak B $ not necessarily effective; but, by Lemma 2.5 of \cite{alessio},
see also Lemma 3.39 of \cite{KollarMori},
and since $- \mathfrak B $ is h-nef, we conclude that $\mathfrak B $ is effective.
We establish the claim by seting $\mathfrak A : = \mathfrak B + h^* C$.

Observe that, under the isomorphism of homology groups, $f^* (K_{Z'} + W')$
corresponds to $(K_{Z} + \hat{X})|_{\hat{X}}= K_{\hat{X}} $, in 
particular we have
$$ K_{\hat{X}}^n = f^* (K_{Z'} + W')^n = (K_W + \mathfrak A)^n . $$
If we assume that $K_W$ is ample, then $ (K_W + \mathfrak A)^n \geq (K_W)^n $,
equality holding if and only if $  \mathfrak A = 0$.

Under assumption (**), it follows that  $$K_{\hat{X}}^n = |G| K_X^n = 
|G| K_{X'}^n= K_W^n,$$
hence  $  \mathfrak A = 0$. Since however $K_W$ is ample, it follows 
that $f$ is an embedding.

If instead we assume that $K_W$ has the same class as $ f^* (K_{Z'} + W')$,
we conclude first that necessarily $\mathfrak B =0$, and then we get that $C= 0$.

Hence $W'$ is normal and $W$ has canonical singularities.

{\bf Step 4}

Finally, the group $G$ acts on $W$, preserving the direct summands of 
its fundamental group
$\Ga$. Hence, $G$ acts on the curve-factors, and the locally symmetric factors.

By assumption, moreover, it sends the summand $\Lambda_i$ to itself, 
hence we get a well defined linear action  on each
Abelian variety $A'_i$, so that we have a diagonal linear action of 
$G$ on $A'$.

Since however the image of $W$ generates $A'$, we can extend the 
action of $G$ on $W$ to a compatible  affine
action on $A'$.

Our final step consists in showing that 

\begin{lemma}\label{affine}
Given a  diagonal special Inoue type manifold, the real affine type of the action of $G$ on the
Abelian variety $ A = (A_1 \times \dots \times A_r)$ is determined by the fundamental group
exact sequence

$$ 1 \ra \Ga= \Pi_{i=1}^r (\Lambda_i) \times   \Pi_{j=1}^h (\pi_{g_j}) 
\times \Pi_{l=1}^s
(\pi_1 (M_l))
\ra \pi_1 ( X) \ra G \ra 1.$$ 
\end{lemma}
{\em Proof.}
Define as before $\Lambda : = \Pi_{i=1}^r (\Lambda_i) = \pi_1 (A)$; moreover, since all the summands in
the left hand side are normal in $ \pi_1 ( X)$, set 
$$\overline{G} : =  \pi_1 ( X) / ( \Pi_{j=1}^h (\pi_{g_j}) 
\times \Pi_{l=1}^s
(\pi_1 (M_l)) ).$$ 

Observe now that $ X$ is the quotient of its universal covering 
$$\tilde{X} = \CC^m \times \prod_{j=1}^h\mathbb{H}_j \times \prod_{l=1}^s\mathcal{D}_l$$ by its fundamental group,
acting diagonally (here $\mathbb{H}_j$ is a copy of Poincar\'e's upper half plane,while $ \mathcal{D}_l$ is an irreducible  bounded symmetric domain
of dimension at least two), hence we  obtain that $\overline{G}$ acts on $\CC^m$ as a group of affine transformations.

Let $\sK$ be the kernel of the associated homomorphism $$ \alpha : \overline{G} \twoheadrightarrow \im (\alpha) =: \hat{G} \subset \Aff (m, \CC),$$
and let $$\overline{G}_1 : = \ker (  \alpha_L : \overline{G} \ra \GL  (m, \CC)).$$

$\overline{G}_1 $ is obviously Abelian, and contains $\Lambda$, and maps onto a lattice $\Lambda' \subset \hat{G}$.

Since $\Lambda$ injects into $\Lambda'$, $\Lambda \cap \sK = 0 $, whence $\sK$ injects into $G$, therefore $\sK$ is a torsion subgroup; since $\Lambda'$ is free, we obtain
$$\overline{G}_1 = \Lambda'  \oplus \sK,$$
and we finally get $$  \sK = \Tors (\overline{G}_1 ), \ \  \hat{G} =  \overline{G}  / \Tors (\overline{G}_1 ).$$

Since our action is diagonal, we can write $\Lambda'  = \oplus_{i=1}^r (\Lambda'_i) $, and the linear action of
the group $G_2 : = G / \sK$ preserves the summands. 

Since $ \hat{G} \subset \Aff (\Lambda')$, we see that $$ \hat{G} = (\Lambda')  \rtimes G'_2 ,$$
where $G'_2 $ is the isomorphic image of $G_2$ inside $ \GL ( \Lambda'))$. 

This shows that the affine group  $ \hat{G}$ is uniquely determined.

Finally, using the image groups $G_{2,i}$ of $G_2$ inside $ \GL ( \Lambda'_i))$,
we can  define uniquely groups of affine transformations of $A_i$ which fully determine
the diagonal action of $G$ on $A$ (up to real affine automorphisms of each $A_i$).

\qed

\QED {\em  for Theorem \ref{special diagonal}}
\bigskip

In order to obtain weak rigidity results, one has to 
introduce  a further invariant, called Hodge type, according to
the following
definition. We shall return to this question  in a  sequel to this paper.

\begin{defin}
   Let $X$ be an  Inoue-type manifold (or variety) of special diagonal 
type, with

$ Z = (A_1 \times \dots \times A_r) \times  (C_1 \times \dots \times
C_h) \times (M_1 \times \dots \times
M_s)$

Then an invariant of the integral representation $ G \ra Aut (\Lambda_i)$
is its {\bf Hodge type},
which is the datum, given the decomposition of $\Lambda_i \otimes \CC$
as the sum of
isotypical components $$\Lambda_i \otimes \CC = \oplus_{\chi \in \Irr (G)}
U_{i, \chi}$$
of the dimensions   $$\nu (i, \chi): = dim_{\CC} U_{i, \chi}  \cap
H^{1,0}(A_i) $$
of the Hodge summands for non real representations.
\end{defin}

\section {Genus $5$ curves having a $(\ZZ/2 \ZZ)^4$-action}\label{section2}

The following is well known (see however section 1 of \cite{burniat1}).

\begin{lemma}\label{genus1}
Let $E_1$ be a compact curve of genus $1$  and assume $ G_1: = (\ZZ/ 2
\ZZ)^n \subset Aut (E_1)$. Then $ n \leq 3$, and, for $n=3$,
$ E_1/ G_1 \cong \PP^1$ with quotient map  branched on exactly
$4$ points $P_1, \dots, P_4$.
The covering  $ E_1 \ra  E_1/ G_1$ factors through multiplication by 
2 in $E_1$.
\end{lemma}

\begin{lemma}\label{genus5}
Let $D$ be a compact curve of genus $5$  and assume $ H : = (\ZZ/ 2
\ZZ)^n \subset Aut (D)$.
Then $ n \leq 4$, and, if $ n = 4$, $ D / H \cong \PP^1$ and the quotient map is branched on exactly
$5$ points $P_1, \dots, P_5$.
\end{lemma}
\proof
By the Hurwitz' formula, one has, setting $ D / H = C$, and setting
$h$= genus ($C$),
$$ 8 = 2^n (2h-2 + \frac{m}{2}) \Leftrightarrow  2^{n-4} (4h-4 + m) = 1$$
where $m$ is the number of branch points $P_1, \dots, P_m$. Hence
$n \leq 4 $. If $n=4$, then $h=1$ is not possible,
since in this case  the  abelianization of  $\pi_1 (C \setminus {P_1}) $ would equal $\pi_1
(C)$, and one would have $m=0$,
a contradiction.
Hence $h=0$ and $m=5$.

\qed

The following geometrical game is based on the fact that the 15
intermediate double covers of
   $ D / H = \PP^1$ are 5 elliptic curves (branched each on 4 of the 5
branch points) and 10 rational curves
(branched each on 2 of the 5 branch points).
Let $A_i$ be the elliptic curve branched on all the five points with
exclusion of $P_i$: then
$D \ra D / H $ factors as $$D \ra  A_i \ra A_i \ra D / H $$
where the middle map is multiplication by 2, and $D \ra  A_i$ is the
quotient by an involution with
   fixed points; the number of  fixed points is exactly 8, since, if $g
\in H$, the fixed set $\Fix(g)$ is an $H$-orbit, and has
therefore cardinality equal to a multiple of $8$. The other 10
involutions have no fixed points, hence they
yield each an unramified covering of a curve $C_j$ of genus $3$.

We try now to stick to Inoue's original notation, except that we refuse to
use the classical symbol for  the {\em Weierstrass $\wp$-function} to denote
the Legendre function $\sL$;   $\sL$ is a  homographic transform of
  the Weierstrass function, but not equal to the Weierstrass function.

  The Legendre function satisfies the quadratic relation (see \cite{burniat1})
  $$ y^2 = (\sL^2 - 1) ( \sL^2 - a^2) .$$

Let $E_1$, $E_2$ be two complex elliptic curves. We assume $E_i = \CC
/ \langle 1, \tau_i \rangle$. Moreover, we denote by
$z_i$ a uniformizing parameter on $E_i$.

  Then $(\ZZ/2 \ZZ)^3 =
\langle \gamma_1, \gamma_2, \gamma_3 \rangle$ acts on
$E_i$ by
\begin{itemize}
   \item[-] $\gamma_1(z_i) = -z_i$;
\item[-] $\gamma_2(z_i) = z_i + \frac 12$;
\item[-] $\gamma_3(z_i) = z_i + \frac{\tau_i}{2}$.
\end{itemize}

 We consider the Legendre $\sL$-function for $E_i$ and denote it by 
$\sL_i$,  for $i=1,2$:
$\sL_i$ is a meromorphic
function on $E_i$ and $\sL_i \colon E_i \ra \PP^1$ is a double cover 
ramified in
$\pm 1, \pm a_i \in \PP^1 \setminus \{0, \infty\}$.

It is well known that we have (cf. \cite{inoue}, lemma 3-2, and also 
cf. \cite{burniat1}, pages 52-54, section 1 for an algebraic 
treatment):
\begin{itemize}
   \item[-] $\sL_i(\frac 12) = -1$, $\sL_i(0) = 1$,
$\sL_i(\frac{\tau_i}{2}) = a_i$,  $\sL_i(\frac{1+\tau_i}{2}) = - 
a_i$;
\item[-] let $b_i:= \sL_i(\frac{\tau_i}{4})$: then $b_i^2 = 
a_i$;
\item[-] $\frac{\rm{d} \sL_i}{\rm{d} z_i} (z_i) = 0$ if and 
only if 
$z_i \in \{0, \frac 12,\frac{\tau_i}{2}, \frac{1+\tau_i}{2} 
\}$, since $\frac { d x_i}{ \sL_i} = dz_i$.
\end{itemize}

Moreover 
$$
\sL_i(z_i) = \sL_i(z_i+1) = \sL_i(z_i+\tau_i) = \sL_i(-z_i) = 
-\sL_i(z_i+ \frac 12),
$$
$$
\sL_i(z_i + \frac{\tau_i}{2}) = 
\frac{a_i}{\sL_i(z_i)}.
$$
We consider now the vector space 
$V_i:=H^0(E_i, \hol_{E_i}(2[0]))$ 
(for $i=1,2$), and note that $V_i 
\cong \CC^2$ with basis
$1, \sL_i$, since 
$$div (1 - \sL_i) = 2 [0] 
-  {\it Poles} (\sL_i) .$$ .

Observe that $[- \frac 14] + [\frac 
14]$ is a $(\ZZ/ 2\ZZ)^2 =\langle 
\gamma_1, \gamma_2 \rangle$ - 
invariant divisor, hence $V_i$
is a $(\ZZ / 2\ZZ)^2$ - module and 
splits in its isotypical components as
$$ V_i = V_i^{++} \oplus 
V_i^{+-},
$$ since $1$ is invariant under $(\ZZ/2\ZZ)^2$ and $\sL_i$ 
is invariant under $\gamma_1$ and is an eigenvector with
eigenvalue 
$-1$ of $\gamma_2$.

If $c \in \CC \setminus \{ \pm 1, \pm a_i, \pm 
a_1a_2 \}$, then the divisor
$$ D_c:=\{ (z_1,z_2) \in E_1 \times E_2 
\ | \ \sL_1(z_1) \sL_2(z_2) = c \}
$$

of bidegree $(2,2)$ is a 
smooth curve of genus $5$. More precisely,
$$
\hol_{E_1 \times 
E_2}(D_c) \cong p_1^* \hol_{E_1}(2[0]) \otimes p_2^* 
\hol_{E_2}(2[0]).
$$

Consider the product action of $(\ZZ/2 \ZZ)^3 
\times (\ZZ/2 \ZZ)^3$ 
on $E_1 \times E_2$.
\begin{rem} 1) It is easy 
to see that $D_c$ is invariant under the subgroup $H 
\leq (\ZZ/2 
\ZZ)^3 \times (\ZZ/2 \ZZ)^3$ given by
$$ (\ZZ/2 \ZZ)^3 \cong H:= 
\langle (\begin{pmatrix} 1\\ 0 \\ 0 
\end{pmatrix},\begin{pmatrix} 
0\\0\\0
\end{pmatrix}),(\begin{pmatrix} 0\\ 0 \\ 0 

\end{pmatrix},\begin{pmatrix} 1\\0\\0 
\end{pmatrix}),(\begin{pmatrix} 
0\\ 1 \\ 
0
\end{pmatrix},\begin{pmatrix} 0\\1\\0 \end{pmatrix}) 
\rangle,
$$

where the coordinates are taken with respect to the 
basis $\gamma_1, 
\gamma_2, \gamma_3$ on each factor.

2) Moreover, 
if we choose $c:= b_1b_2$, then we see that for 
$(z_1,z_2) \in 
D_{b_1b_2}$:
$$
\sL_1(z_1 + \frac{\tau_1}{2})\sL_2(z_2 + 
\frac{\tau_2}{2}) = 
\frac{a_1a_2}{\sL_1(z_1)\sL_2(z_2)} 
=
\frac{a_1a_2}{b_1b_2} = b_1b_2,
$$ whence $D_{b_1b_2}$ is invariant 
under
$$ (\ZZ / 2 \ZZ)^4 \cong G:= H \oplus \langle (\begin{pmatrix} 
0\\ 0 
\\ 1 \end{pmatrix},\begin{pmatrix} 0\\0\\1
\end{pmatrix}) \rangle \leq (\ZZ/2 \ZZ)^3 \times (\ZZ/2 \ZZ)^3.
$$
\end{rem}

We want to show that the converse holds. More precisely, we prove the following
\begin{prop}\label{genus5,2}
   Let $f \colon D \ra \PP^1$ be the maximal $G:= (\ZZ/2
\ZZ)^4$-covering branched in $5$ given points $p_1, \ldots , p_5 \in
\PP^1$. Then there are two elliptic curves $E_1$, $E_2$ such that $D
\subset E_1 \times E_2$ is a $G$-invariant divisor with
$$
\hol_{E_1 \times E_2}(D) \cong p_1^* \hol_{E_1}(2[0]) \otimes p_2^*
\hol_{E_2}(2[0]).
$$ Choosing appropriate coordinates we can moreover assume that $D=
\{ (z_1,z_2) \in E_1 \times E_2 \ | \ \sL_1(z_1)
\sL_2(z_2) = b_1b_2 \}$.
\end{prop}
\begin{proof}
   Let $e_1, e_2,e_3,e_4$ be a basis of the $\ZZ/2 \ZZ$ - vectorspace
$G$ and let $D \rightarrow \PP^1$ branched in $p_1, \ldots
p_5$ be given by the {\em appropriate orbifold homomorphism}
$$
\varphi : \mathbb{T}(2,2,2,2,2) := \langle x_1, \ldots , x_5 |
\prod_{i=1}^5 x_i, x_1^2, \ldots, x_5^2 \rangle \rightarrow
(\ZZ/2\ZZ)^4,
$$ where $\varphi(x_i) = e_i$ for $1\leq i \leq4$, $\varphi(x_5) =
e_5:=e_1 +e_2+e_3+e_4$. Then Hurwitz' formula shows that
$D$ is a smooth curve of genus $5$.

Note that the only elements of $G$ having fixed points on $D$ are the 5 elements 
$e_i$ ($1 \leq i \leq 5$), and each of them has exactly $8$
fixed points on $D$. Hence, $E_i:=D/\langle e_i\rangle$ is an
elliptic curve.

We get therefore $5$ elliptic curves (as intermediate covers of $D
\rightarrow \PP^1$), all endowed with a $(\ZZ/2 \ZZ)^3$ - action.

Choose two of these elliptic curves, say $E_1, E_2$, and consider the morphism
$$ i \colon D \rightarrow D/\langle e_1\rangle \times D/\langle
e_2\rangle = E_1 \times E_2.
$$

Then $D \cdot E_i =2$ and $i$ is an embedding of $D$ as a $(\ZZ / 2
\ZZ)^4$ - invariant divisor of bidegree $2$.

We fix the origin in both elliptic curves so that $D= \{s=0 \}$, where $s \in H^0(E_1 \times E_2,
p_1^*(\hol_{E_1}(2[0]))) \otimes
p_2^*(\hol_{E_2}(2[0])))^G$.

It remains to show that we can assume $D$ to be of the form $\{
(z_1,z_2) \in E_1 \times E_2 \ | \ \sL_1(z_1) \sL_2(z_2) =
b_1b_2 \}$.

For this we show that
$$H^0(E_1 \times E_2, p_1^*(\hol_{E_1}(2[0]))) \otimes
p_2^*(\hol_{E_2}(2[0])))^H \cong \CC^2.$$ In fact,

\begin{multline*}
   H^0(E_1 \times E_2, p_1^*(\hol_{E_1}(2[0]))) \otimes
p_2^*(\hol_{E_2}(2[0]))) = V_1 \otimes V_2 = \\ = (V_1^{+++} \oplus
V_1^{++-}) \otimes (V_2^{+++} \oplus V_2^{++-}),
\end{multline*} whence
\begin{multline*} H^0(E_1 \times E_2, p_1^*(\hol_{E_1}(2[0])))
\otimes p_2^*(\hol_{E_2}(2[0])))^H = \\ = (V_1^{+++}
\otimes V_2^{+++}) \oplus (V_1^{++-} \otimes V_2^{++-}).
\end{multline*}

Therefore we have a pencil of  $H$-invariant divisors $D_c:=\{
(z_1,z_2) \in E_1 \times E_2 \ | \ \sL_1(z_1) \sL_2(z_2) = c
\}$. It is now obvious that $D_c$ is $G$ - invariant iff $c = \pm b_1b_2$.
The change of sign for $b_i$ is achieved by changing the point $\frac{\tau_i}{4}$
with $\frac{\tau_i}{4} + \frac{1}{2}$.
\end{proof}

Consider now for the moment  the action of $(\ZZ /
2\ZZ)^2 = \langle \gamma_1, \gamma_2\rangle$ on $E_1$ and $E_2$, and
the induced product action on $E_1 \times E_2$. Assume $H \cong (\ZZ
/ 2 \ZZ)^3 \leq (\ZZ / 2\ZZ)^2 \times (\ZZ / 2\ZZ)^2$,
such that
\begin{enumerate}
  \item on each factor $E_i$ it induces the action 
of $(\ZZ / 2\ZZ)^2= 
\langle \gamma_1, \gamma_2 \rangle$;
\item there 
is an $H$-invariant pencil in $\PP(V_1 \otimes V_2)$, 
i.e., $\dim 
(V_1 \otimes V_2)^H = 2$.
\end{enumerate}

Then it is easy to see 
that we can choose $(\gamma_1, 0)$, 
$(0,\gamma_1)$, $(\gamma_2, 
\gamma_2)$ as basis of $H$.

Therefore
$$ (V_1 \otimes V_2)^{+++} = 
(V_1^{+++} \otimes V_2^{+++}) \oplus 
(V_1^{++-} \otimes 
V_2^{++-}),
$$ and the $H$- invariant pencil is given by $D_c = 
\{(z_1,z_2) \in 
E_1 \times E_2 \ | \ \sL_1(z_1) \sL_2(z_2) = c 
\}$.

Consider now $G:= H \oplus \langle (\gamma_3, \gamma_3) \rangle 
\cong 
(\ZZ/2 \ZZ)^4$. 

Then $D:= \{(z_1,z_2) \in E_1
\times E_2 \ | \ 
\sL_1(z_1) \sL_2(z_2) = b_1b_2\}$.

\begin{rem} \
\begin{enumerate}
 
\item The restriction map  $$ H^0(E_1 \times E_2, p_1^*\hol_{E_1}(2[0]) \otimes p_2^* 
\hol_{E_2}) \rightarrow H^0(D, p_1^*\hol_{E_1}(2[0]) \otimes
p_2^* 
\hol_{E_2} | D)
$$ is an isomorphism of $H$-modules.
\item There is a 
pencil of such $H$- invariant divisors of degree 4 
on $D$. But there 
is no one which is invariant under $(\ZZ /2
\ZZ)^4$.
\end{enumerate}
\end{rem}

\begin{proof}[Proof 
of rem. ( 1).] Let be $S:= E_1 \times E_2$ and for 
simplicity we 
write
$$
\hol_S(2,0) := p_1^*\hol_{E_1}(2[0]) \otimes p_2^* 
\hol_{E_2}.
$$ Then we consider the exact sequence:
$$ 0 \rightarrow 
\hol_S(-D) \otimes  \hol_S(2,0) \rightarrow 
\hol_S(2,0) \rightarrow 
\hol_S(2,0) \otimes \hol_D \rightarrow 0.
$$ By K\"unneth's formula 
we get
\begin{itemize}
  \item[i)] $h^0(\hol_S(-D) \otimes 
\hol_S(2,0)) = 
h^0(\hol_{E_1})h^0(\hol_{E_2}(-2) = 0$;
\item[ii)] 
$h^1(\hol_S(-D) \otimes  \hol_S(2,0)) = 2$.
\end{itemize} 
Therefore
$$ r \colon H^0(S, \hol_S(2,0)) \cong \CC^2 \rightarrow 
H^0(S, 
\hol_S(2,0) \otimes \hol_D)
$$ is injective. Since $D$ is not 
hyperelliptic it follows by 
Clifford's theorem that $h^0(S, 
\hol_S(2,0) \otimes \hol_D) \leq 2$.
This implies that $h^0(S, 
\hol_S(2,0) \otimes \hol_D) = 2$ and $r$ is 
an isomorphism (of 
$H$-modules).
\end{proof}

\begin{rem}
  This implies that
$$ H^0(S, 
\hol_S(2,0) \otimes \hol_D) \cong \CC^2 = V^{+++} \oplus 
V^{++-}.
$$
\end{rem}

\section{Inoue surfaces with $p_g=0$ and 
$K_S^2 = 7$}

In \cite{inoue} the author describes the construction 
of a family of 
minimal surfaces of general type $S$ with $p_g= 0$, 
$K_S^2
= 7$.

We briefly recall the construction of these surfaces 
and, for lack of 
reference, we calculate $K_S^2$ and $p_g(S)$.

For 
$i \in \{1,2,3,4\}$, let $E_i := \CC / \langle 1, \tau_i \rangle$ 
be 
a complex elliptic curve. Denoting again by $z_i$ a
uniformizing 
parameter of $E_i$, we consider the following five 
involutions on 
$T:=E_1 \times E_2\times E_3 \times E_4$:
\begin{itemize}
\item[] $ g_1(z_1,z_2,z_3,z_4) = 
(-z_1+\frac 12, z_2 + \frac 12, z_3,z_4)$,
\item[] $g_2(z_1,z_2,z_3,z_4) 
= (z_1, -z_2 + \frac 12, z_3 +\frac 12,-z_4+ \frac 12)$,
\item[] $g_3(z_1,z_2,z_3,z_4) = (z_1+\frac 12, z_2, -z_3 + \frac 12,-z_4  + 
\frac 12)$,
\item[]  $g_4(z_1,z_2,z_3,z_4) = (z_1, z_2, -z_3,-z_4)$,
\item[] $g_5(z_1,z_2,z_3,z_4) = (z_1+\frac{\tau_1}{2}, 
z_2+\frac{\tau_2}{2}, 
z_3+\frac{\tau_3}{2},z_4+\frac{\tau_4}{2})$.
\end{itemize}
 Then $G:= \langle g_1,g_2,g_3,g_4,g_5 \rangle \cong (\ZZ / 2 \ZZ)^5$.

Consider
$$
\hat{X} := \{ (z_1,z_2,z_3,z_4)) \in T \ | \ \sL_1(z_1)
\sL_2(z_2)\sL_3(z_3)  = b_1b_2b_3, \ \sL_3(z)\sL_4(z_4) =
b_3b_4\}.
$$ Then
\begin{itemize}
   \item[-] $\hat{X}$ is a smooth complete intersection of two
hypersurfaces in $T$ of respective multidegrees $(2,2,2,0)$ and
$(0,0,2,2)$;
\item[-] $\hat{X}$ is invariant under the action of $G$, and $G$ acts
freely on $\hat{X}$.
\end{itemize}

The above equations show that  $\hat{X}$ is the complete intersection of a $G$ -
invariant divisor $X_1$ in the linear system
$$
\PP(H^0(E_1 \times E_2 \times E_3 \times E_4, p_1^*\hol_{E_1}(2[0])
\otimes p_2^*\hol_{E_2}(2[0]) \otimes
p_3^*\hol_{E_3}(2[0]) \otimes p_4^*\hol_{E_4})),
$$  with a $G$-invariant divisor $X_2$ ($\cong E_1 \times E_2 \times
D$), where $D$ is a curve of genus $5$, a Galois cover
with group $(\ZZ/ 2 \ZZ)^4$ of the projective line  ramified in 5
points) in the linear system
$$
\PP(H^0(E_1 \times E_2 \times E_3 \times E_4, p_1^*\hol_{E_1} \otimes
p_2^*\hol_{E_2} \otimes p_3^*\hol_{E_3}(2[0])
\otimes p_4^*\hol_{E_4}(2[0]))).
$$

\begin{rem}
   It is easy to see from the above explicit description that $G$ acts
freely on $\hat{X}$. Note that the action of $G$ has fixed points
on $T$, and even on $X_2$.
\end{rem}

\begin{defin} A smooth projective algebraic surface $S:= \hat{X} / G$ as above
is called an {\em Inoue surface} with $K_S^2 =7$ and
$p_g=0$.

\end{defin} 
 \begin{rem}
   It is immediate to see that the involution $g_4$ acts freely on  $ \hat{X} $
   and trivially on $(E_1 \times E_2)$.
   
    It follows therefore that,
    setting $ \hat{X}_2 : =  \hat{X} / g_4 $,
    $S$ has another
   representation as $ \hat{X}_2 / (\ZZ / 2 \ZZ)^4$, where
   $ \hat{X}_2$ is a divisor in the product $(E_1 \times E_2 \times C)$,
  and $C$ is a smooth curve of genus $3$.
   
This was the representation of $S$ announced in \cite{survey}, and it follows from our results that
this representation is also unique.

\end{rem}

We will show in the next lemma that $S$ is a minimal
surface of general type with $K_S^2 = 7$ and $p_g=0$.

\begin{lemma}\label{invariants}
  Let $S=\hat{X} /G$ be as above. 
Then $S$ is a minimal surface of 
general type with $K_S^2 = 7$ and 
$p_g=0$.
\end{lemma}

\begin{proof}

Since $\hat{X}$ is of general 
type, also $S$ is of general type being 
an \'etale quotient of 
$\hat{X}$ by a finite group.

We first remark that $\hat{X}$ is a 
$G$-invariant hypersurface of 
multidegree $(2,2,4)$ in $W:=E_1 
\times E_2 \times D$,
where $D \subset E_3\times E_4$ is a smooth 
curve of genus 5 given by 
the equation
$$
\{ (z_3,z_4))\in E_3 
\times E_4 \ | \ \sL_3(z)\sL_4(z_4) = b_3b_4\}.
$$

By the adjunction 
formula, the canonical divisor of $\hat{X}$ is the 
restriction to 
$\hat{X}$ of a divisor of multidegree $(2,2,12)$
on $W$.

Therefore 
we can calculate (denoting by $F_i$ the fibre of the 
projection of 
$W$ on the $(j,k)$-th coordinate (with $\{i,j,k\} 
=
\{1,2,3\}$).
\begin{multline*} K_{\hat{X}}^2 = ((K_W + 
[\hat{X}])|S)^2 = 
(2F_1+2F_2+12F_3)^2(2F_1+2F_2+4F_3)=\\ = 
(8F_1F_2+
48F_1F_3 + 48 F_2F_3)(2F_1+2F_2+4F_3)= \\ = 32F_1F_2F_3 + 
96F_1F_2F_3+96F_1F_2F_3 = 224 = 7 \cdot 2^5.
\end{multline*}

Since $G$ acts freely on $\hat{X}$, we obtain
$$ 
K_S^2 = \frac{224}{|G|} = \frac{224}{2^5} = 7.
$$

Moreover,  
consider the exact sequence
\begin{equation}\label{adjunction}
  0 
\rightarrow \hol_W(K_W) \rightarrow \hol_W(K_W + [\hat{X}]) 
\rightarrow \omega_{\hat{X}} \rightarrow 0.
\end{equation}

Using 
K\"unneth's formula and Kodaira's vanishing theorem, we 
get:
\begin{itemize}
\item[-] $\dim H^0(W, \hol_W(K_W)) = 
5$,
\item[-] $\dim H^0(W, \hol_W(K_W+ [X])) = 32$,
\item[-] $H^i(W, 
\hol_W(K_W+ [X])) = 0$, for $i=1,2,3$,
\item[-] $\dim H^1(W, 
\hol_W(K_W)) = 1 +5+5 = 11$,
\item[-] $\dim H^2(W, \hol_W(K_W)) = 
q(W) (= q(\hat{X})) = 7$.
\end{itemize} Therefore, by the long exact 
sequence associated to 
(\ref{adjunction}) we get:
$$ p_g(\hat{X}) = 
h^0(\hat{X}, \omega_{\hat{X}}) = 32+11-5 = 38.
$$ 
Therefore
$$
\chi(\hol_{\hat{X}}) = 1+p_g(\hat{X}) -q(\hat{X}) = 1 + 
38 -7 =32.
$$

This implies that $\chi(\hol_S) = 1$. In order to show 
that $p_g(S) = 
0$, it suffices to show that $q(S) =0$. Using the 
fact that
$$ i^* \colon H^0(W, \Omega^1_W) \rightarrow H^0(\hat{X}, 
\Omega^1_{\hat{X}})
$$ is an isomorphism and that
$$ H^0(W, 
\Omega^1_W) = H^0(E_1, \Omega^1_{E_1})\oplus H^0(E_2, 
\Omega^1_{E_2}) \oplus H^0(D, \Omega^1_D),
$$ it is easy to see that 
$H^0(W, \Omega^1_W)^G = H^0(S, 
\Omega^1_S)=0$.
\end{proof}

\subsection{The torsion group of Inoue surfaces with $K_S^2 =7$}
The aim of the section is to prove the following
\begin{theo}\label{homology}
Let $S$ be an Inoue surface with $K_S^2 = 7$. Then 
$$
H_1(S, \ZZ) \cong \ZZ/4\ZZ \times (\ZZ /2 \ZZ)^4.
$$
\end{theo}
It is clear from the construction that the fundamental group of an Inoue surface sits in an exact sequence
\begin{equation}\label{fundex}
1 \rightarrow \mathbb{Z}^4 \times \pi_5 \rightarrow \pi_1(S)
\rightarrow G \cong (\mathbb{Z}/2 \mathbb{Z})^5
\rightarrow 1,
\end{equation}
 where $\pi_5$ denotes the fundamental group of a compact curve
 of genus five.
 
 Observe that after moding out by $\ZZ^4$ in the exact sequence (\ref{fundex}) we obtain the orbifold exact sequence (plus a summand $\ZZ/2 \ZZ$) of the maximal $(\ZZ / 2 \ZZ)^4$-covering of $\PP^1$ ramified in 5 points:
 \begin{equation}\label{orbifold}
 1 \ra \pi_5 \ra \pi_1^{orb} \times \ZZ/2 \ZZ  \ra G \cong (\mathbb{Z}/2 \mathbb{Z})^4 \times \ZZ/2 \ZZ \ra 1,
 \end{equation}
 where $\pi_1^{orb}:=\pi_1^{orb}(\PP^1 \setminus\{p_1, \ldots , p_5\};2,2,2,2,2)$.
 
 The proof of the above theorem will be divided in several steps, and we shall first 
 prove  some auxiliary results.
\begin{lemma}\label{action}
Assume that there is an exact sequence of groups
$$
1 \ra \Lambda \ra \Gamma' \ra G \ra 1,
$$
where $\Lambda$ and $G$ are abelian. 

Assume moreover that
\begin{itemize}
\item[($\#$)]  $\Lambda$ admits a a system $L$ of generators with the following property:
 $ \forall h \in L$  $ \exists g \in G$ such that 
$$
ghg^{-1} = -h.
$$
\end{itemize}
Then $2 \Lambda \subset [\Gamma', \Gamma']$ and, in particular, we have an exact sequence 
\begin{equation}\label{lem1}
\Lambda/ 2 \Lambda \ra \Gamma'^{ab} \ra G \ra 1.
\end{equation}
\end{lemma}

\begin{proof}
Let $h \in \Lambda$, $g \in G$ be such that $ghg^{-1} = -h$. Then $[g,h] = -2h$, whence $2h \in [\Gamma', \Gamma']$. 
Since this holds for all $h \in L$ and $L$ generates $\Lambda$ the claim follows.

\end{proof}

\begin{rem}\label{reduction}
i) Assume that there is an exact sequence of groups
$$
1 \ra H \ra \Gamma \ra G \ra 1,
$$
where $G$ is abelian. Defining $\Gamma' := \Gamma / [H,H]$,  we have an exact sequence
\begin{equation}\label{exrem}
1 \ra \Lambda:= H^{ab} \ra \Gamma' \ra G \ra 1,
\end{equation}
and we have
$$
(\Gamma')^{ab} = \Gamma^{ab}.
$$
Suppose now that assumption ($\#$) of lemma \ref{action} is satisfied for the exact sequence (\ref{exrem}). Then by lemma \ref{action} we have an exact sequence 
\begin{equation}
\tilde{\Lambda} :=  \Lambda/2 \Lambda \ra \Gamma^{ab} \ra G \ra 1.
\end{equation}

\noindent 
ii) Define $\Gamma'':= \Gamma' / 2 \Lambda$. Then we have an exact sequence
$$
1 \ra \tilde{\Lambda} \ra \Gamma'' \ra G \ra 1,
$$
and $(\Gamma'')^{ab} = (\Gamma')^{ab} = \Gamma^{ab}$.

Choose generators $(g_i)_{i \in I}$ of $G$ and choose for each $g_i$ a lift $\gamma_i$ to $\Gamma''$. Moreover, let $(\lambda_j)_{j \in J}$ be generators of $\Lambda$ and denote their images in $\tilde{\Lambda}$ by $\tilde{\lambda}_j$. Then obviously $\Gamma'' = \langle (g_i)_{i \in I}, (\tilde{\lambda}_j)_{j \in J} \rangle$.

Assume also that 
\begin{itemize}
\item[($\# \#$)] $\gamma_i \lambda_j \gamma_i^{-1} = \pm \lambda_j$, $\forall i \in I, \ \forall j \in J$.
\end{itemize}
Then (since $2\tilde{\Lambda} = 0$) we have $[\gamma_i, \tilde{\lambda}_j] = 0$ (in $\Gamma''$), $\forall i \in I, \ \forall j \in J$.

 In particular this implies that 
$$
\Gamma'' / \langle [\gamma_i, \gamma_j], i,j \in I \rangle = (\Gamma'')^{ab} = \Gamma^{ab}.
$$
\end{rem}

Since the genus 5 curve $D$ is an ample divisor $D \subset E_3 \times E_4$, we have by Lefschetz' theorem a surjective morphism
 $$
\varphi:  \pi_5 \cong \pi_1(D) \rightarrow \pi_1(E_3 \times E_4) \cong \ZZ^4.
 $$
Defining  $K:= \ker \varphi$ and  $\Lambda:= (\pi_5 \times \ZZ^4)/ (K \times \{0\}) \cong \ZZ^4 \times \ZZ^4$, we get the exact sequence
\begin{equation}\label{prop1}
1 \rightarrow \Lambda \rightarrow \pi_1(S) / K \rightarrow (\ZZ / 2 \ZZ)^5 \rightarrow 1.
\end{equation}

\begin{prop} 
$P:=(\pi_1(S) / K)^{ab} \cong  \ZZ/4\ZZ \times (\ZZ /2 \ZZ)^4$.
\end{prop}

\begin{proof}
We first verify condition ($\#\#$) for the exact sequence (\ref{prop1}). 

For this we recall the description of  the action of $G:= \langle g_1,g_2,g_3,g_4,g_5 \rangle \cong (\ZZ / 2 \ZZ)^5$ on 
$T:=E_1 \times E_2\times E_3 \times E_4$:
\begin{itemize}
\item[] $ g_1(z_1,z_2,z_3,z_4) = 
(-z_1+\frac 12, z_2 + \frac 12, z_3,z_4)$,
\item[] $g_2(z_1,z_2,z_3,z_4) 
= (z_1, -z_2 + \frac 12, z_3 +\frac 12,-z_4+ \frac 12)$,
\item[] $g_3(z_1,z_2,z_3,z_4) = (z_1+\frac 12, z_2, -z_3 + \frac 12,-z_4  + 
\frac 12)$,
\item[]  $g_4(z_1,z_2,z_3,z_4) = (z_1, z_2, -z_3,-z_4)$,
\item[] $g_5(z_1,z_2,z_3,z_4) = (z_1+\frac{\tau_1}{2}, 
z_2+\frac{\tau_2}{2}, 
z_3+\frac{\tau_3}{2},z_4+\frac{\tau_4}{2})$.
\end{itemize}
 
 Observe that in this situation $\Lambda = \pi_1 (T) \subset \CC^4$.

We claim now that for each generator $\lambda \in  \{\tau_1 e_1, \tau_2 e_2, \tau_3 e_3, \tau_4e_4, e_1,  e_2,e_3, e_4 \}$ there is a $g \in G$
 such that $g \lambda g^{-1} = - \lambda$.
 
  For this it suffices to find, for each $k \in \{1,2,3,4\}$, an element $g \in G$ whose linear action has $e_k$ as $-1$ eigenvector.
  
 But this is obvious (e.g. $(g_1g_3, g_1g_2,g_4,g_4)$ for $k=(1,2,3,4)$). 
 
 Lemma \ref{action} implies now that, setting $P':=\pi_1(S)/ K$, then  
 $$
 2\Lambda \subset [P', P'].
 $$
Denoting by $\gamma_1, \ldots , \gamma_5$ lifts to $P'$ of the generators $g_1, \ldots, g_5$ of $G$, we know that
$$
P' = \langle \gamma_1, \ldots , \gamma_5, \lambda_1, \ldots , \lambda_4 \rangle.
$$
Moreover, $[\gamma_i, \lambda_j] = \pm 2\lambda_j$, whence by remark \ref{reduction}, we have an exact sequence 
$$1 \ra \Lambda /2 \Lambda \ra P'':=P'/2 \Lambda \ra G \ra 1,
$$
and $P = (P')^{ab}=(P'')^{ab} = P'' / \langle [\gamma_i, \gamma_j], 1 \leq i,j \leq 5\rangle$.

A straightforward computation shows the following:

$[\gamma_1,\gamma_2] = \begin{pmatrix}
0\\1\\0\\0
\end{pmatrix}$, $[\gamma_1,\gamma_3] = \begin{pmatrix}
-1\\0\\0\\0
\end{pmatrix}$, $[\gamma_1,\gamma_4] = 0$, $[\gamma_1,\gamma_5] = \begin{pmatrix}
- \tau_1\\0\\0\\0
\end{pmatrix}$,

$[\gamma_2,\gamma_3] = \begin{pmatrix}
0\\0\\1\\0
\end{pmatrix}$, $[\gamma_2,\gamma_4] = \begin{pmatrix}
0\\0\\1\\1
\end{pmatrix}$, $[\gamma_2,\gamma_5] = \begin{pmatrix}
0\\-\tau_2\\0\\-\tau_4
\end{pmatrix}$, $[\gamma_3,\gamma_4] = \begin{pmatrix}
0\\0\\1\\1
\end{pmatrix}$,

$[\gamma_3,\gamma_5] = \begin{pmatrix}
0\\0\\-\tau_3\\-\tau_4
\end{pmatrix}$, $[\gamma_4,\gamma_5] = \begin{pmatrix}
0\\0\\-\tau_3\\-\tau_4
\end{pmatrix}$.

This immediately implies that 
$$
\tilde{\Lambda} /  \langle [\gamma_i, \gamma_j] \rangle \cong \ZZ/2 \ZZ
$$
hence we get the exact sequence 
$$
1\ra \ZZ/2 \ZZ \ra P \ra G = (\ZZ/2 \ZZ)^5 \ra 1.
$$

Since 
$$
g_5^2 = \begin{pmatrix}
\tau_1\\ \tau_2 \\ \tau_3\\ \tau_4
\end{pmatrix},
$$
it follows that $P \cong \ZZ/4\ZZ \times (\ZZ /2 \ZZ)^4$.

\end{proof}

Finally we prove
\begin{prop}
The natural surjective morphism 
$$H_1(S, \ZZ) = \pi_1(S)^{ab} \onto P
$$ is an isomorphism.
\end{prop}
This implies theorem \ref{homology}.

\begin{proof}
Let $\Gamma:= \pi_1(S)$. Then from the exact sequence
$$
1 \rightarrow H:=\mathbb{Z}^4 \times \pi_5 \rightarrow \Gamma
\rightarrow G \cong (\mathbb{Z}/2 \mathbb{Z})^5
\rightarrow 1,
$$
we get by remark \ref{reduction} the exact sequence 
$$
1 \ra H^{ab} \cong \ZZ^{10} \times \ZZ^4 =: \Lambda_1 \oplus \Lambda_2 \ra \Gamma' \ra G \ra 1,
$$
and $\Gamma^{ab} =(\Gamma')^{ab}$. 

By lemma \ref{action} we get an exact sequence 
$$
\ZZ^{10} \times (\ZZ/2 \ZZ)^4 \ra \Gamma'' \rightarrow G \rightarrow 1,
$$
where again $(\Gamma'')^{ab} = \Gamma^{ab}$. 

Note that $g_1$ acts trivially on the curve $D$ of genus 5, whence it acts trivially on $H_1(D) \cong \ZZ^{10}$. 

Moreover, we have seen that the commutators $[\ga''_1, 
\ga''_i]$, $2 \leq i \leq 5$, span a subspace $V$ of rank 3 in $\Lambda_2 / 2 \Lambda_2 \cong (\ZZ/2 \ZZ)^4$. Therefore we get (after moding out by $V$) an exact sequence
$$
 \ZZ^{10} \times \ZZ/2 \ZZ \ra \Gamma''' \rightarrow G \rightarrow 1,
$$
where $(\Gamma''')^{ab} = \Gamma^{ab}$. After dividing by $\ZZ/2\ZZ$, we obtain the exact sequence 
$$
H_1(D, \ZZ) \rightarrow \pi := \Gamma'''/(\ZZ / 2 \ZZ)  \ra G \ra 1.
$$
We finally get the following commutative diagram with exact rows
\begin{equation}\label{diagram}
\xymatrix{
\ZZ/ 2 \ZZ\ar[r]&\Gamma'''\ar[d] \ar[r] & \pi\ar[d] \ar[r]&1\\
\ZZ/ 2 \ZZ\ar[r]&(\Gamma''')^{ab}\ar[r]&(\pi)^{ab} = (\pi_1^{orb})^{ab} \oplus \ZZ / 2 \ZZ \cong (\ZZ / 2 \ZZ)^5\ar[r]&1,
}
\end{equation}
(cf. the exact sequence (\ref{orbifold})). 

This implies that $H_1(S) = \Gamma ^{ab} = (\Gamma''')^{ab}$ has cardinality at most $2 \cdot 2^5 = 2^6$. Since $|P| = 2^6$, 
the surjective morphism $H_1(S, \ZZ) \onto P$ is then an isomorphism. 
\end{proof}

\section{Weak rigidity  of Inoue surfaces}

In this section we  shall prove the following:

\begin{theo}\label{homotopy}
    Let $S'$ be a smooth complex projective surface which is
homotopically equivalent to an Inoue surface (with $K^2 = 7$ and
$p_g=0$). Then $S'$ is an Inoue  surface.

The same consequence holds under the weaker assumptions that $S'$ has the same fundamental group of an Inoue surface,
and that (2) (SAME HOMOLOGY) of theorem \ref{special diagonal} holds.

\end{theo}

\begin{rem}
   It is clear from the definition that  $S$ is a diagonal CIT (classical Inoue   type) manifold.
   
   We can therefore apply theorem \ref{special diagonal} to this special case.
    In this special case, we are going to see that the groups $\hat{G}_i = \Lambda'_i \rtimes G_{2,i}$, $i=1,2$, are obtained simply by taking $  \Lambda'_i : = \frac{1}{2}  \Lambda_i$, and $ G_{2,i} : = \{ \pm 1\}$.
    
     We shall indeed view things geometrically, as follows.
   
  Recall  that there is an exact sequence
$$ 1 \rightarrow \mathbb{Z}^4 \times \pi_5 \rightarrow \pi_1(S)
\rightarrow G \cong (\mathbb{Z}/2 \mathbb{Z})^5
\rightarrow 1,
$$ where $\pi_5$ denotes the fundamental group of a compact curve
 of genus five.
\end{rem}

\noindent Let $\hat{X} \subset E_1 \times E_2 \times D$ be the
\'etale $(\ZZ /2\ZZ)^5$-covering of $S$. Observe that $H_1:=
\langle g_2,g_3,g_4,g_5 \rangle$ acts trivially on $H^0(E_1,
\Omega^1_{E_1})$. This shows that $q(\hat{X} /G) \geq 1$, and
$$ q(\hat{X} /H_1) = 1 \ \ \iff \ \ D/H_1 \cong \mathbb{P}^1.
$$ But this is obvious, since $D/G = D/H_1$.

Now, consider instead $H_2 := \langle g_1,g_3,g_4,g_5 \rangle$, which
acts trivially on $H^0(E_2, \Omega^1_{E_2})$.
Therefore
$$ q(\hat{X} /H_2) = 1 \ \ \iff \ \ D/H_2 \cong \mathbb{P}^1.
$$

$D/H_2 \cong \mathbb{P}^1$ follows from the following lemma
\ref{quot}, since $H_2$ contains three elements having fixed
points on $D$ (in fact, the elements  $(-z_3 + \frac 12,-z_4 + \frac
12)$, $(-z_3 + \frac{\tau_3}{2},-z_4 +
\frac{\tau_4}{2})$,$(-z_3 + \frac 12 +\frac{\tau_3}{2},-z_4 + \frac 
12+\frac{\tau_4}{2})$).

Therefore we have 
seen:

\begin{prop}\label{irreg}
  Let $S:=\hat{X} / G$ be an Inoue 
surface. Then there are subgroups 
$H_1, H_2 \leq G$ of index 2, such 
that $q(\hat{X}/H_1) =
q(\hat{X}/H_2) = 1$.

Let $S'$ be a smooth 
complex projective surface which has the same 
fundamental group as 
$S$. Then, denoting by $\hat{X}'$ the
corresponding \'etale 
$G$-covering of $S'$ we have:
\begin{itemize}
  \item there is a 
smooth curve $D'$ of genus 5 and a surjective 
homomorphism $\hat{X}' 
\rightarrow D'$;
\item there are two index two subgroups $H_1, H_2$ 
of $G$ such that 
$q(\hat{X}/H_1) = q(\hat{X}/H_2) = 
1$.
\end{itemize}

\end{prop}

\begin{lemma}\label{quot}

Let $D \rightarrow \mathbb{P}^1$ be the maximal $\Gamma \cong
(\mathbb{Z} / 2 \mathbb{Z})^4$ -covering branched in five
points. Then:

1) there are exactly 5 subgroups $H \leq \Gamma$, $H \cong
(\mathbb{Z}/2 \mathbb{Z})^3$ containing exactly one element
having fixed points on $D$ ($\implies$ $g(D/H) = 1$);

2) there are exactly 10 subgroups $H \leq \Gamma$, $H \cong
(\mathbb{Z}/2 \mathbb{Z})^3$ containing exactly three elements
having fixed points on $D$ ($\implies$ $g(D/H) = 0$).
\end{lemma}

\begin{proof} Let $D \rightarrow \mathbb{P}^1$ be the maximal $\Gamma
\cong (\mathbb{Z} / 2 \mathbb{Z})^4$ -covering
branched in five points $p_1, \ldots p_5$, which determines a
surjective homomorphism
$$
\varphi \colon \pi_1(\mathbb{P}^1 \setminus \{p_1, \ldots p_5 \})
\rightarrow \Gamma.
$$

We denote $\varphi(\gamma_i)$, where $\gamma_i$ is a geometric loop
around $p_i$ by $e_i$. Then $e_1, e_2,e_3, e_4$ is a
$\mathbb{Z}/ 2\mathbb{Z}$ - basis of $\Gamma$ and $e_1 + e_2 +e_3+e_4 = e_5$.

\begin{claim}
   Let $H \cong (\ZZ /2 \ZZ)^3 \leq \Gamma$. Then either there is a
unique $i \in \{ 1, \ldots , 5 \}$ such that $e_i \in H$ (and
$e_j \notin H$ for $j \neq i$, or there is a subset $\{i,j,k \}
\subset \{ 1, \ldots , 5 \}$ such that $e_i, e_j, e_k \in H$ (and the
other two $e_l$'s are not in $H$).
\end{claim}

{\em Proof of the claim.} It is clear that $H$ contains at
least one of the $e_i$'s, otherwise $D \rightarrow D/H$ is
\'etale. By Hurwitz' formula, we get
$$ 8 = 2g(D) -2 = |H|(2(g(D/H)-2) = 8(2g(D/H)-2),
$$ a contradiction. Since any four of the $e_i$'s are linearly
independent, $H$ can contain at most three of them.

Assume now that there are $i \neq j$, such that $e_i, e_j \in H$.
W.l.o.g. we can assume $e_1, e_2 \in H$. Then $H = \langle e_1,
e_2, 
\lambda_1e_1 +\lambda_2e_2 +\lambda_3e_3 +\lambda_4e_4 \rangle$, 

$\lambda_i \in \{0,1\}$. Note that $e_1, e_2,
\lambda_1e_1 
+\lambda_2e_2 +\lambda_3e_3 +\lambda_4e_4$ are linearly 
independent 
if and only if $(\lambda_3,
\lambda_4) \neq (0,0)$. 
Moreover,
\begin{itemize}
\item $e_3 \notin H  \iff 
(\lambda_3,\lambda_4) \neq (1,0)$,
\item $e_4 \notin H  \iff 
(\lambda_3,\lambda_4) \neq (0,1)$,
\item $e_5 \notin H  \iff 
(\lambda_3,\lambda_4) \neq (1,1)$.
\end{itemize} This shows: if $H$ 
contains two of the $e_i$'s then 
also a third one.

\QED

Now, 
if $H$ contains three of the $e_i$'s, say $e_i, e_j,e_k$, then 
$H = 
\langle e_i, e_j, e_k \rangle$, and there are exactly
$\binom{5}{3} = 
10$ such subgroups.

The remaining $5$ subgroups $H \cong (\ZZ / 2 
\ZZ)^3$ of $\Gamma$ 
contain therefore exactly one of the 
$e_i$'s.

Assume now that
$$ H_1 := \langle e_1, e_2, e_3 \rangle, \ 
\ H_2:= \langle e_1+e_2, 
e_1+e_3, e_1+e_4 \rangle.
$$

Then it 
remains to show that $g(D/H_1)=0$ and $g(D/H_2) = 1$.

Observe that 
$H_2 = \langle e_5 \rangle \oplus \langle e_1 +e_2, 
e_1+e_3 
\rangle$, and $H':=\langle e_1 +e_2, e_1+e_3
\rangle$ acts freely on 
$D$. Therefore $g(D/H') = 2$ and $e_5$ acts 
on $D/H'$ having two 
fixed points. By Hurwitz' formula this
implies that $g(D/H) 
=1$.

Now, $g(D/ \langle e_1 \rangle) = 1$ and $H_1/ \langle e_1 
\rangle 
\cong \langle e_2, e_3\rangle$ acts with fixed points on
$D/ 
\langle e_1 \rangle$. This shows that $D/H_1 \cong 
\PP^1$.

\end{proof}

Now we are ready to finish the proof of  \ref{homotopy}.

\begin{proof}(of thm. \ref{homotopy})

Let $S'$ be a 
smooth complex projective surface which is homotopically equivalent 
to an Inoue surface $S$ with $K_S^2 =7$ and
$p_g=0$. 

In particular 
$\pi_1 (S) \cong \pi_1(S')$ and we take the \'etale $G:=(\ZZ / 
2\ZZ)^5$- covering $\hat{X}'$, which
is homotopically equivalent to 
$\hat{X}$ (the corresponding covering 
of the Inoue surface $S$).  By 
proposition \ref{irreg} we
know that $\hat{X}'$ admits a morphism to 
a curve $D'$ of genus 5, 
and there are subgroups $H_1, H_2 \leq G$ 
of index 2 such
that $X_i:=\hat{X}'/H_i$ has irregularity 
one.

Therefore there are elliptic curves
$E_1', E_2'$ and 
morphisms
$$
\hat{X}' \rightarrow X_i \rightarrow E'_i.
$$ By the 
universal property of the Albanese map we get a commutative 
diagram

\begin{equation}\label{diagr}
\xymatrix{
\hat{X}'\ar[r]\ar[d]&E_1' 
\times E_2' \times \Jac(D')\\
\Alb(\hat{X}')\ar_{\psi}[ru]&\\ 
}
\end{equation}

\begin{lemma} Let $E_i' = \mathbb{C}/\Lambda_i'$ 
and denote by 
$\Lambda_i:= 2\Lambda_i'$. Then $\psi$ corresponds to 
$$
H_1(D', \ZZ) \times \Lambda_1 \times \Lambda_2 \subset H_1(D', 
\ZZ) 
\times \Lambda_1' \times \Lambda_2'.
$$

In particular $E_i:= 
\mathbb{C}/\Lambda_i = E_i'$, $\psi$ 
restricted to $E_i$ is 
multiplication by 2, and
$\Alb(\hat{X}') = \Jac(D') 
\times E_1 \times E_2.$

\end{lemma} 

 By theorem \ref{special diagonal} we have that
$\varphi \colon \hat{X}' \rightarrow 
 E_1 \times 
E_2 \times D' $ is a birational morphism onto its image 
$W'
\subset  E_1 \times 
E_2 \times D'  : = C_1 \times C_2 \times C_3$.

 In fact, since the fundamental group of $X'$ is isomorphic to the
one of a Inoue surface, it follows by lemma \ref{affine} that $G$ acts on $Z'$ as for an Inoue surface, hence
 there is no effective divisor $\De$ of 
numerical type $(1,1,2)$ which is invariant by the action of the group $G$, as it is easy to verify.

Therefore  $W'$ has homology class $2F_1 + 2F_2 
+ 4F_3$, where $F_i$ is the fibre over a point in the i-th curve, $C_i$
and $W'$ has rational double points as singularities.

The linear equivalence class of $W'$ is invariant for the group action. It is the sum of three classes $\xi_i$ which
are respective pull backs from the projection
onto the i-th curve $C_i$. Hence each  class $\xi_i$ is invariant for the action of $G$ on $C_i$,
hence $\xi_i$ is the pull back from the quotient of $C_i$ by the group $G_i$, projection of $G$ into the automorphism group
of $C_i$.

By our lemmas, \ref{genus1} and \ref{genus5}, these quotients are rational curves, hence
we conclude that the linear equivalence class of the divisor $W'$ is the same as the one for
an Inoue surface $S$.

It remains to show that $W'$ is given by Inoue's equations, i.e., if we consider the genus 5 curve $D'$ as a hypersurface 
$$
D' := \{ z_3,z_4)) \in E_3 \times E_4 \ |  \ \sL_3(z)\sL_4(z_4) =
b_3b_4\}, 
$$
then 
\begin{multline}
W' := \{ (z_1,z_2,z_3,z_4)) \in E_1 \times  \ldots \times E_4 \ | \ \sL_1(z_1)
\sL_2(z_2)\sL_3(z_3)  = b_1b_2b_3, \\
 \sL_3(z_3)\sL_4(z_4) =
b_3b_4\}.
\end{multline}
For this consider the subgroup $H:=\langle g_1,g_2,g_3,g_4 \rangle \leq G$. Then there is a $H$-invariant divisor with class $2F_1+2F_2 +4F_3$ in $Z':=E_1 \times E_2 \times D'$. Therefore $H^0(Z', \hol_Z'(2F_1+2F_2+4F_3)) \cong \CC^8$ is an $H$-module which decomposes by K\"unmeth's formula and the results in section \ref{section2} as follows:

\begin{multline}
H^0(Z', \hol_Z'(2F_1+2F_2+4F_3)) \cong \\
\cong H^0(E_1 \times E_2 \times D', p_1^*(\hol_{E_1}(2[0]))
\otimes p_2^*(\hol_{E_2}(2[0]))\otimes p_3^*(\hol_{D'}(2,0))) \cong \\
\cong V_1 \otimes V_2 \otimes V_3 \cong (V_1^{++++} \oplus V_1^{-+-+}) \otimes  (V_2^{++++} \oplus V_2^{--++}) \otimes (V_3^{++++} \oplus V_3^{+--+}) \cong \\
\cong V^{++++} \oplus V^{+--+} \oplus V^{--++} \oplus V^{-+-+},
\end{multline}
where each of the four summands in the last line is isomorphic to $\CC^2$. In fact, we have

$$V^{++++} = (V_1^{++++} \otimes V_2^{++++} \otimes V_3^{++++} ) \oplus (V_1^{-+-+} \otimes V_2^{--++} \otimes V_3^{+--+} ),
$$
$$V^{+--+} = (V_1^{++++} \otimes V_2^{++++} \otimes V_3^{+--+} ) \oplus (V_1^{-+-+} \otimes V_2^{--++} \otimes V_3^{++++} ),
$$
$$V^{--++} = (V_1^{++++} \otimes V_2^{--++} \otimes V_3^{++++} ) \oplus (V_1^{-+-+} \otimes V_2^{++++} \otimes V_3^{+--+} ),
$$
$$V^{-+-+} = (V_1^{-+-+} \otimes V_2^{++++} \otimes V_3^{++++} ) \oplus (V_1^{++++} \otimes V_2^{--++} \otimes V_3^{+--+} ).
$$
The equations of the hypersurfaces in the above pencils are then:
$$
W'_1(c) := \{ c = \sL_1(z_1) \sL_2(z_2) \sL_3(z_3) \},
$$
$$
W'_2(c) := \{ c = \sL_1(z_1) \sL_2(z_2) \frac{1}{\sL_3(z_3)} \},
$$
$$
W'_3(c) := \{ c = \sL_1(z_1) \frac{1}{\sL_2(z_2)} \sL_3(z_3) \},
$$
$$
W'_4(c) := \{ c = \frac{1}{\sL_1(z_1)} \sL_2(z_2) \sL_3(z_3) \}.
$$

This shows that after possibly replacing one of the elliptic curves $E_i$ with parameter $a_i$ by the elliptic curve $E'_i$ with parameter $\frac{1}{a}$ we can w.l.o.g. assume that the pencil of $H$-invariant hypersurfaces in $Z'$ is given by the equation $\{ c = \sL_1(z_1) \sL_2(z_2) \sL_3(z_3) \}$.

Now, we consider $g_5$. It is easy to see that $g_5(W'_1(c)) \equiv W'_1(c)$, and if $W'_1(c) = div(s)$ for $s \in V^{++++}$, then also $g_5(W'_1(c)) = div(s')$ for $s' \in V^{++++}$. Therefore $g_5$ is an involution on $\PP^1 := \PP(V^{++++})$, which is obviously non trivial,  whence $g_5$ has exactly two fixed points. Therefore there are exactly two $G$-invariant divisors in the pencil $W'_1(c)$. On the other hand,
$W'_1(b_1b_2b_3)$ and $W'_1(-b_1b_2b_3)$ are $G$-invariant. 

This shows that $W'$ is of the desired  form, hence $X'$ is the canonical model of an Inoue surface.

\end{proof}

\section{Inoue surfaces as bidouble 
covers and $H^1(S,\Theta_S)$}

The aim of this section is to show the 
following 
\begin{theo}
Let $S$ be an Inoue surface with $K_S^2 = 7$. 
Then:
$$
h^1(S, \Theta_S) = 4, \ \  h^2(S, \Theta_S) = 
8.
$$
\end{theo}

To prove this result we resort to a result of 
\cite{mlp}, where   Inoue surfaces are constructed as bidouble covers of the 
four nodal cubic.

We briefly recall their description here, for 
details we refer to 
\cite{mlp}, example 4.1 (we keep their notation, even if slightly inconvenient).

We consider a complete 
quadrilateral $\Lambda$ in $\PP^2$  and 
denote the vertices by $P_1, \ldots,
P_6$.

We have labeled the 
vertices in a way that
\begin{itemize}
  \item the intersection point 
of the line $\overline{P_1P_2}$ and the line 
$\overline{P_3P_4}$ is $P_5$,
\item the 
intersection point of $\overline{P_1P_4}$ and $\overline{P_2P_3}$ is 
$P_6$.
\end{itemize}

Let $Y \rightarrow \PP^2$ be the blow up in 
$P_1, \ldots , P_6$, 
denote by $L$  the total transform of a line in 
$\PP^2$, let $E_i$, $1
\leq i \leq6$, be the exceptional curve lying over 
$P_i$.
Moreover, we denote by $S_i$, $1 \leq i \leq 4$, the strict 
transforms on $Y$ of the sides $S_i : = \overline{P_i P_{i+1}}$ for $ 1\leq i \leq 3$,
$S_4 : = \overline{P_4 P_{1}}$, of the quadrilateral $\Lambda$.

The geometry of the situation is that the four (-2) curves $S_i$ come from the resolution
of the 4 nodes of the cubic surface $\Sigma$ which is the anticanonical image of $Y$,
and the curves $E_i$ are the strict transforms of the 6 lines in $\Sigma$ connecting pairs of nodal points.

The surface $\Sigma$ contains also a triangle of lines (joining the midpoints of opposite edges
of the tetrahedron with sides the lines corresponding to the curves $E_i$. These are the 3 strict transforms 
 $\Delta_1$, $\Delta_2$,
$\Delta_3$ of the three diagonals of the complete quadrilateral $\Lambda$. $\Delta_1$
is the strict transform of $\overline{P_1P_3}$, $\Delta_2$ of $\overline{P_2P_4}$ and $\Delta_3$
of $\overline{P_5P_6}$.

For each line $\De_i$ in the cubic surface  $\Sigma$ we consider the pencil of planes containing
them, and the base point free pencil of residual conics, which we denote by $| f_i|$. Hence we have
$$| f_i| = |(-K_Y) -  \De_i| , \ \  \De_i + f_i \equiv (-K_Y) .$$

In the plane realization we have:
\begin{itemize}
  \item $f_1$ is the strict 
transform on $Y$ of a general element of
the pencil of conics $\Gamma_1$ through
$P_2,P_4,P_5,P_6$,
\item $f_2$ is the strict transform on $Y$ of a general element of
the pencil of conics $\Gamma_2$ through $P_1,P_3,P_5,P_6$,
\item $f_3$ is the strict transform on $Y$ of a general element of
the pencil of conics $\Gamma_3$ through $P_1,P_2,P_3,P_4$.
\end{itemize}

It is then easy to see that each curve $S_h$ is disjoint from the other curves $S_j$ ($j\neq h$), $\De_i$, and $f_i$
if $f_i$ is smooth. Moreover, 
$$\De_i \cdot f_i = 2,  \  \  \De_i \cdot f_j = 0, i \neq j, \ \ f_i ^2 = 0 , \ \  f_j f_i = 2,  i \neq j . $$

\begin{definition}\label{Inouedivisors}
   We define the {\em Inoue divisors} on $Y$ as follows:
\begin{itemize}
\item $D_1:= \Delta_1 + f_2 + S_1 + S_2$, where $f_2 \in |f_2|$ smooth;
\item $D_2:= \Delta_2 + f_3$, where $f_3 \in |f_3|$ smooth;
\item$D_3:= \Delta_3 + f_1 +f_1'+ S_3 + S_4$, where $f_1, f_1' \in
|f_1|$ smooth.
\end{itemize}

\end{definition} 

Let $\pi \colon \tilde{S} \rightarrow Y$ be the
bidouble covering with branch divisors $D_1, D_2, D_3$ (associated to the 
3 nontrivial elements of the Galois group.

 By the previous remarks we see
that over each $S_i$ there are two disjoint $(-1)$-curves.
Contracting these eight exceptional curves we obtain a minimal surface
with $p_g=0$ and $K_S^2 = 7$. By \cite{mlp} these are exactly the Inoue surfaces.

\begin{rem}\label{ratvar}
We immediately see that there is an open dense subset  in the product 
$$
|f_1| \times |f'_1| \times |f_2| \times |f_3| \cong (\PP^1)^4
$$
parametrizing the family of Inoue surfaces.
\end{rem}

\begin{rem}
   The non trivial character sheaves of this bidouble cover are
\begin{itemize}
   \item $\mathcal{L}_1 = \hol_Y( - K_Y + f_1 - E_4)$;
\item $\mathcal{L}_2 = \hol_Y( - 2 K_Y  - E_5 - E_6)$;
\item $\mathcal{L}_3 = \hol_Y( - K_Y + L - E_1 - E_2  - E_3)$.
\end{itemize}

\end{rem}

\begin{lemma}\label{inv} \
\begin{itemize}
    \item $\dim H^1(\tilde{S}, \Theta_{\tilde{S}})^{inv} = \dim H^1(S,
\Theta_S)^{inv} = 4$, 
\item  $\dim H^2(\tilde{S}, \Theta_{\tilde{S}})^{inv} = \dim H^2(S,
\Theta_S)^{inv} = 0$.
\end{itemize}
\end{lemma}

\begin{proof}
It is well known (cf. e.g.  \cite{cime2}) that  $H^2(\tilde{S}, \Theta_{\tilde{S}})$ of the $(\ZZ/2\ZZ)^2$-covering  $\pi \colon \tilde{S} \ra Y$ decomposes as a direct sum of character spaces
$$
H^2(\tilde{S}, \Theta_{\tilde{S}}) \cong H^2(\tilde{S}, \Theta_{\tilde{S}})^{inv} \oplus \bigoplus_{i=1}^3 H^2(\tilde{S}, \Theta_{\tilde{S}})^i,
$$
and that the dimensions of the direct summands can be computed as the dimensions of global sections of logarithmic differential forms on the basis $Y$. In fact, we have:
\begin{multline}\label{chars}
h^0(Y, \Omega^1_Y(\log D_1, \log D_2, \log D_3)(K_Y)) = h^2(\tilde{S},\Theta_{\tilde{S}})^{inv} = h^2(S,\Theta_S)^{inv}; \\
h^0(Y, \Omega^1_Y(\log D_i)(K_Y + \mathcal{L}_i))) = h^2(\tilde{S},\Theta_{\tilde{S}})^i = h^2(S,\Theta_S)^i, \forall i \in \{1,2,3\}.
\end{multline}
Note that $|-K_{Y}| = |3L - \sum_{i=1}^6 E_i|$ 
is non empty 
and does not have a fixed part. Therefore there is a 
morphism
$\hol_{Y}(K_{Y}) \ra \hol_{Y}$, 
which is not 
identically zero on any component of the 
$D_i$'s.

We get the 
commutative diagram with exact 
rows
\begin{equation}\label{diagram2}
\xymatrix{ 0\ar[r]& 
\Omega^1_{Y}(K_{Y}) \ar[d] 
\ar[r]&\Omega^1_{Y}((\log D_1, \log D_2, \log D_3)(K_{Y})\ar[d]\ar[r]
&\oplus_{i=1}^3 
\hol_{D_i}(K_{Y}) \ar[d] \ar [r]&0\\ 
0\ar[r]&\Omega^1_{Y} \ar[r]&\Omega^1_{Y}((\log D_1, \log D_2, \log D_3) 
\ar[r] &\oplus_{i=1}^3 \hol_{D_i}  \ar [r]&0.}
\end{equation}

 From 
this we get the commutative 
diagram with injective vertical arrows
\begin{equation*}
\xymatrix{
\CC^2 \oplus 0 \oplus\CC^2 
\cong&H^0(Y,\oplus_{i=1}^3
\hol_{D_i}(K_{Y}))\ar[r]^{\delta}\ar[d]^{\psi_2}\ar[dr]^{\varphi}& 
H^1(Y,
\Omega^1_{Y}(K_{Y})) \ar[d] \\
\CC^4 
\oplus \CC^2 \oplus\CC^5 \cong &H^0(Y,\oplus_{i=1}^3 
\hol_{D_i})\ar[r]^{\psi_1} &H^1(Y,
\Omega^1_{Y}) 
.}
\end{equation*} 

A usual argument shows that (see \cite{catjdg}) $\delta$ is injective. In fact, the Chern classes of $S_1,S_2,S_3,S_4$ are linearly independent, hence  $\varphi$ is injective, which implies that also
$\delta$ is 
injective. Therefore
$$ h^0(\Omega^1_{Y}((\log D_1, \log D_2, \log D_3)(K_{Y})) 
= 
h^2(\Theta_{\tilde{S}})^{inv} = h^2(\Theta_S)^{inv} = 
0.
$$

Therefore
\begin{multline}
  h^1(\Theta_{\tilde{S}})^{inv} = - 
\chi(\Omega^1_{Y}((\log D_1, \log D_2, \log D_3)(K_{Y})) = \\
=
-(\chi(\Omega^1_{Y}(K_{Y}) + \chi(\oplus_{i=1}^3 
\hol_{D_i}(K_{Y}))).
\end{multline}

An easy calculation shows now that 
$\chi(\oplus_{i=1}^3 
\hol_{D_i}(K_{Y}))) = 0$, 
whereas
$\chi(\Omega^1_{Y}(K_{Y})) = - 4$.
\end{proof}

We prove now

\begin{prop}\label{h1theta}  \
\begin{enumerate}
\item $h^0(Y, \Omega^1_Y(\log D_1)(K_Y + \mathcal{L}_1)) = 
h^2(\tilde{S}, \Theta_{\tilde{S}})^1 = h^2(S, \Theta_S)^1 \leq 2$;
\item $h^0(Y, \Omega^1_Y(\log D_2)(K_Y + \mathcal{L}_2)) = 
h^2(\tilde{S}, \Theta_{\tilde{S}})^2 = h^2(S, \Theta_S)^2 \leq 3$;
\item $h^0(Y, \Omega^1_Y(\log D_3)(K_Y + \mathcal{L}_3)) = 
h^2(\tilde{S}, \Theta_{\tilde{S}})^3 = h^2(S, \Theta_S)^3 \leq 3$.
\end{enumerate}
In particular, we get $h^2(S,\Theta_S) \leq 8$.
\end{prop}

From this result, we easily obtain the following:
\begin{cor}\label{cortheta}
Let $S$ be an Inoue surface with $K_S^2 =7$, $p_g=0$. Then:
$$
h^2(S, \Theta_S) = 8, \ \ h^1(S, \Theta_S)= h^1(S, \Theta_S)^{inv} = 4.
$$

\end{cor}

\begin{proof}[Proof of the cor..]
We have by proposition \ref{h1theta} 
$$
8 - h^1(S,\Theta_S) \geq h^2(S, \Theta_S) - h^1(S,\Theta_S) = \chi(\Theta_S) = 2K_S^2 - 10 \chi(S) = 4,
$$
whence $h^1(S, \Theta_S) \leq 4$. On the other hand by lemma \ref{inv} we know that
$$
4 = h^1(S, \Theta_S)^{inv} \leq h^1(S, \Theta_S)
$$
and the assertions of the corollary follow.
\end{proof}
\begin{proof}[Proof of proposition \ref{h1theta}.]
1) Recall that by definition \ref{Inouedivisors} $D_1$ is the 
disjoint union of the four curves $\Delta_1, f_2, S_1, S_2$, and
$$
K_Y + \mathcal{L}_1 \equiv  f_1 -E_4.
$$

We consider the exact sequence (cf. e.g.  \cite{esnaultviehweg}, p. 13)
\begin{multline}\label{exseq}
0 \ra \Omega^1_Y(\log \Delta_1, \log f_2, \log S_1, \log S_2, \log f_1)(-E_4) \ra \\
\ra \Omega^1_Y((\log \Delta_1, \log f_2, \log S_1, \log S_2)(f_1-E_4)) \ra \\
\ra \Omega^1_{f_1} (\Delta_1 +f_2 + S_1+S_2+f_1-E_4) \ra 0.
\end{multline}

Since
$$
(\Delta_1 +f_2 + S_1+S_2+f_1-E_4)f_1 = 2+2-1,
$$
we have $ \Omega^1_{f_1} (\Delta_1 +f_2 + S_1+S_2+f_1-E_4) \cong \hol_{\PP^1}(1)$, whence 
\begin{multline}
h^0(Y, \Omega^1_Y(\log D_1)(K_Y+ \mathcal{L}_1)) = h^0(Y, 
\Omega^1_Y((\log \Delta_1, \log f_2, \log S_1, \log S_2)(f_1-E_4)) \leq \\
 \leq h^0(Y, \Omega^1_Y(\log \Delta_1, \log f_2, \log S_1, \log S_2, \log f_1)(-E_4)) + 2.
\end{multline}
Consider the long exact cohomology sequence of the short exact sequence

\begin{multline}
0 \ra \Omega^1_Y \ra  \Omega^1_Y(\log \Delta_1, \log f_2, \log S_1, \log S_2, \log f_1) \ra \\
\ra  \hol_{f_1} \oplus \hol_{f_2} \oplus \hol_{S_1} \oplus \hol_{S_2} 
\oplus \hol_{\Delta_1} \ra 0
\end{multline}

Since $H^0(Y, \Omega^1_Y)=0$, 
$H^0(Y, \Omega^1_Y(\log \Delta_1, \log f_2, \log S_1, \log S_2, \log f_1))$ is the kernel of the connecting 
homomorphism
$$
\delta \colon H^0(Y, \hol_{f_1} \oplus \hol_{f_2} \oplus \hol_{S_1} \oplus \hol_{S_2} 
\oplus \hol_{\Delta_1}) \ra H^1(Y, \Omega^1_Y).
$$
By \cite{catjdg} the image of $\delta$ is generated by the Chern 
classes of $ \Delta_1, f_2,  S_1, S_2, f_1$.

\begin{claim}
$\dim \im (\delta) = 5$.
\end{claim}
\begin{proof}[Proof of the claim.] 
Assume that
\begin{equation}\label{zero1}
\lambda_1 f_1 + \lambda_2 f_2 + a_1 S_1 + a_2 S_2 + \mu \Delta_1\equiv 0,
\end{equation}
where $\lambda_1, \lambda_2, a_1, a_2, \mu \in \CC$.
Intersection with $E_4$ gives $\lambda_1 =0$, whereas intersection with $S_i$, $i=1,2$ yells $-2a_i=0$. The equation \ref{zero1} has become
$$
\mu \Delta_1 + \lambda_2 f_2 =0.
$$

Intersection with e.g. $E_5$ gives $\lambda_2 =0$, whence also $\mu = 0$.
\end{proof}

It follows now that
$$
h^0(Y, \Omega^1_Y(\log \Delta_1, \log f_2, \log S_1, \log S_2, \log f_1)(-E_4)) \leq 
$$
$$
\leq h^0(Y, \Omega^1_Y(\log \Delta_1, \log f_2, \log S_1, \log S_2, \log f_1)) = 0.
$$

Therefore we have proven
$$
h^0(Y, \Omega^1_Y(\log D_1)(K_Y+ \mathcal{L}_1)) \leq 2.
$$
2) By definition \ref{Inouedivisors} $D_2$ is the disjoint union of 
the two curves $\Delta_2, f_3$, and
$$
K_Y + \mathcal{L}_2 \equiv  -K_Y -E_5-E_6.
$$
Since
$$
(K_Y + 2 \Delta_2 +(- K_Y -E_5-E_6)) \Delta_2 = (2 \Delta_2  -E_5-E_6) \Delta_2 -2 < 0,
$$
by  lemma 5.1 of  \cite{burniat3} we have
$$
H^0(Y, \Omega^1_Y(\log D_2)(K_Y+ \mathcal{L}_2)) \cong H^0(Y, 
\Omega^1_Y(\log f_3)(K_Y + \mathcal{L}_2 + \Delta_2)).
$$
Note that
$$
K_Y + \mathcal{L}_2 + \Delta_2 \equiv -K_Y -E_5 - E_6 + \Delta_2 
\equiv S_1 + S_2 +S_4 + S_3 + E_1 + E_3.
$$

Again by lemma 5.1 in \cite{burniat3} we see that
\begin{multline}
H^0(Y, \Omega^1_Y(\log f_3)(K_Y + \mathcal{L}_2 + \Delta_2)) \cong \\
\cong H^0(Y, \Omega^1_Y(\log f_3, \log S_1, \log S_2, \log S_3, \log 
S_4)(E_1 + E_3)).
\end{multline}
Since $E_1(f_3 + S_1 +S_2 + S_3 +S_4 + E_1 +E_3) = E_3(f_3 + S_1 +S_2 
+ S_3 +S_4 + E_1 +E_3) = 2$, we see by the same arguments (using the analogous exact sequence \ref{exseq}) as in case 
1) that

\begin{multline}
h^0(Y, \Omega^1_Y(\log f_3, \log S_1, \log S_2, \log S_3, \log 
S_4)(E_1 + E_3)) \leq \\
\leq h^0(Y, \Omega^1_Y(\log f_3, \log S_1, \log S_2, \log S_3, \log 
S_4, \log E_1, \log E_3)) +2.
\end{multline}
\begin{claim}
$\dim \langle f_3, S_1, \ldots , S_4, E_1, E_3 \rangle =6$.
\end{claim}
\begin{proof}[Proof of the claim.]
Note that 
$$
S_1+S_4-S_2-S_3 \equiv -2E_1 + 2E_3.
$$
Therefore we have proven the claim, if we show that the Chern classes of $f_3, S_1, \ldots , S_4, E_1$ are linearly independent.

Assume that 
\begin{equation}\label{zero2}
\lambda f_3 + a_1 S_1 + a_2 S_2 + a_3 S_3 +a_4 S_4 + \mu E_1\equiv 0,
\end{equation}
where $\lambda, a_1, a_2, a_3,a_4, \mu \in \CC$.

Intersection with $S_2, S_3$ gives $a_2 = a_3 =0$, whereas intersection with $\Delta_1$ yields $\mu = 0$. We are left with the equation $\lambda f_3 + a_1 S_1 +a_4 S_4 \equiv 0$. Intersection with $E_5$ resp. $E_6$ implies that $a_1=0$ resp. $a_4=0$, and we conclude that also $\lambda=0$.

\end{proof}
Therefore we get
$$
h^0(Y, \Omega^1_Y(\log f_3, \log S_1, \log S_2, \log S_3, \log S_4, 
\log E_1, \log E_3)) = 1,
$$
and we have shown 2).

3)  $D_3$ is the disjoint union of the five curves $\Delta_3, f_1, 
f'_1,S_3, S_4$, and
$$
K_Y + \mathcal{L}_3 \equiv L -E_1-E_2-E_3.
$$
Since
$$
(K_Y + 2 \Delta_3 +(L -E_1-E_2-E_3)) \Delta_3 = -2 < 0,
$$
by \cite{burniat3}, lemma 5.1, we have
\begin{multline}
H^0(Y, \Omega^1_Y(\log D_3)(K_Y+ \mathcal{L}_2)) \cong \\
\cong H^0(Y, \Omega^1_Y(\log f_1, \log f'_1, \log S_3, \log S_4))(K_Y 
+ \mathcal{L}_2 + \Delta_3)).
\end{multline}
Note that
$$
K_Y + \mathcal{L}_3 + \Delta_3 \equiv S_1 + S_2  + E_2.
$$

Again by lemma 5.1 in \cite{burniat3} we see that
\begin{multline}
H^0(Y, \Omega^1_Y(\log f_1, \log f'_1, \log S_3, \log S_4))(K_Y + 
\mathcal{L}_2 + \Delta_3)) \cong \\
\cong H^0(Y, \Omega^1_Y(\log f_1, \log f'_1, \log S_3, \log S_4, \log 
S_1, \log S_2)(E_2)).
\end{multline}
Since $E_2(f_1+f'_1+S_1 +S_2 + S_3 +S_4 + E_2) =  3$, we see by the 
same arguments as in case 1) that

\begin{multline}
h^0(Y, \Omega^1_Y(\log f_1, \log f'_1, \log S_3, \log S_4, \log S_1, 
\log S_2)(E_2)) \leq \\
\leq h^0(Y, \Omega^1_Y(\log f_1, \log f'_1, \log S_3, \log S_4, \log 
S_1, \log S_2, \log E_2)) + 2.
\end{multline}

\begin{claim}
The Chern classes of $f_1, S_1, \ldots , S_4, E_2$ are linearly 
independent.
\end{claim} 

\begin{proof}[Proof of the claim.]

Assume that 
\begin{equation}\label{zero3}
\lambda f_1 + a_1 S_1 + a_2 S_2 + a_3 S_3 +a_4 S_4 + \mu E_2\equiv 0,
\end{equation}
where $\lambda, a_1, a_2, a_3,a_4, \mu \in \CC$.

Intersection with $S_3, S_4$ gives $a_3 = a_4 =0$, whereas intersection with $E_1$ yields then  $a_1=0$.
 Intersection with $E_3$ instead gives  $a_2 =0$.
Finally,  intersection with $E_4$ yields $\lambda =0$, whence also $\mu =0$. \end{proof}
Therefore
$$
h^0(Y, \Omega^1_Y(\log f_1, \log f'_1, \log S_3, \log S_4, \log S_1, 
\log S_2, \log E_2)) = 1,
$$
and we have shown 3).

\end{proof}

\begin{proof}[Proof of theorem \ref{main}]
(1) has been proved in theorem \ref{homotopy}.

(2) By \cite{inoue}, page 318, $K_S$ is ample and by corollary \ref{cortheta} the tangent space $H^1(S, \Theta_S)$ to the base $\Def(S)$ of the Kuranishi family of $S$ consists of the invariants for the action of the group $(\ZZ/2 \ZZ)^2$. Therefore all the local deformations of $S$ admit a $(\ZZ/2 \ZZ)^2$-action, hence are $(\ZZ/2 \ZZ)^2$ Galois coverings of the four nodal cubic surface (the anticanonical image of $Y$).

Furthermore, the dimension of $H^1(S, \Theta_S)$ is equal to the dimension of the Inoue  family containing $S$ in the moduli space $\mathfrak{M}^{can}_{1,7}$, whence the base of the Kuranishi family of $S$ is smooth.

Since the quotient of a smooth variety by a finite group (in our case, the automorphism group  $\Aut(S)$) is normal, it follows that the irreducible connected component of the moduli space $\mathfrak{M}^{can}_{1,7}$ corresponding to Inoue surfaces with $K_S^2=7, p_g = 0$ is normal and in particular generically smooth.

The family of Inoue surfaces is parametrized by a smooth (4-dimensional) rational variety (cf. e.g. remark \ref{ratvar}), whence  unirationality follows.
\end{proof}


\bigskip
\noindent {\bf Authors' Adresses:}

\noindent I.Bauer, F. Catanese \\ Lehrstuhl Mathematik VIII\\
Mathematisches Institut der Universit\"at Bayreuth\\ NW II\\
Universit\"atsstr. 30\\ 95447 Bayreuth

\begin{verbatim}
ingrid.bauer@uni-bayreuth.de,
fabrizio.catanese@uni-bayreuth.de
\end{verbatim}
\end{document}